\documentclass{amsart}

\usepackage[top=2.5cm, bottom=2.5cm, left=2.5cm, right=2.5cm]{geometry}
\usepackage{subfig}
\usepackage{tikz-cd}
\usepackage{tikz}
\usepackage{caption}
\usepackage{amssymb}
\usepackage[all]{xy}
\usepackage{amsthm,amsmath}
\usepackage{mathrsfs}
\usepackage{hyperref}
\hypersetup{
    colorlinks=true,
    filecolor=magenta,      
    urlcolor=cyan,
    bookmarks=true,
}

\usepackage{listings}
\usepackage{pgfplots}

\usepackage{graphicx}

\usepackage{caption}
    \DeclareCaptionLabelFormat{nolabel}{}
    \usepackage{adjustbox} 
    \usepackage{xcolor} 
    \usepackage{enumerate} 
    \usepackage{geometry} 
 \usepackage{textcomp} 
    \AtBeginDocument{%
    }
    \usepackage{upquote} 
    \usepackage{eurosym} 
    \usepackage[mathletters]{ucs} 
    \usepackage[utf8x]{inputenc} 
    \usepackage{fancyvrb} 
    \usepackage{grffile} 
    
    \usepackage{longtable} 
    \usepackage{booktabs}  
    \usepackage[inline]{enumitem} 
    \usepackage[normalem]{ulem} 

    \definecolor{urlcolor}{rgb}{0,.145,.698}
    \definecolor{linkcolor}{rgb}{.71,0.21,0.01}
    \definecolor{citecolor}{rgb}{.12,.54,.11}
 
    \definecolor{ansi-black}{HTML}{3E424D}
    \definecolor{ansi-black-intense}{HTML}{282C36}
    \definecolor{ansi-red}{HTML}{E75C58}
    \definecolor{ansi-red-intense}{HTML}{B22B31}
    \definecolor{ansi-green}{HTML}{00A250}
    \definecolor{ansi-green-intense}{HTML}{007427}
    \definecolor{ansi-yellow}{HTML}{DDB62B}
    \definecolor{ansi-yellow-intense}{HTML}{B27D12}
    \definecolor{ansi-blue}{HTML}{208FFB}
    \definecolor{ansi-blue-intense}{HTML}{0065CA}
    \definecolor{ansi-magenta}{HTML}{D160C4}
    \definecolor{ansi-magenta-intense}{HTML}{A03196}
    \definecolor{ansi-cyan}{HTML}{60C6C8}
    \definecolor{ansi-cyan-intense}{HTML}{258F8F}
    \definecolor{ansi-white}{HTML}{C5C1B4}
    \definecolor{ansi-white-intense}{HTML}{A1A6B2}
    \definecolor{ansi-default-inverse-fg}{HTML}{FFFFFF}
    \definecolor{ansi-default-inverse-bg}{HTML}{000000}

    \DefineVerbatimEnvironment{Highlighting}{Verbatim}{commandchars=\\\{\}}

    \let\Oldtex\TeX
    \let\Oldlatex\LaTeX
    \renewcommand{\TeX}{\textrm{\Oldtex}}
    \renewcommand{\LaTeX}{\textrm{\Oldlatex}}
    \title{Untitled5}

    \hypersetup{
      breaklinks=true, 
      colorlinks=true,
      urlcolor=urlcolor,
      linkcolor=linkcolor,
      citecolor=citecolor,
      }

\DeclareMathOperator{\N}{\mathbf{N}}
\DeclareMathOperator{\Z}{\mathbf{Z}}
\DeclareMathOperator{\R}{\mathbf{R}}

\DeclareMathOperator{\dd}{\mathbf{d}}
\DeclareMathOperator{\ee}{\mathbf{e}}

\DeclareMathOperator{\Hom}{\textup{Hom}}
\DeclareMathOperator{\GL}{\textup{GL}}

\DeclareMathOperator{\Ext}{\textup{Ext}}
\DeclareMathOperator{\C}{\mathbf{C}}
\DeclareMathOperator{\Q}{\mathbf{Q}}
\DeclareMathOperator{\F}{\mathbf{F}}

\DeclareMathOperator{\HH}{\mathbf{H}}

\DeclareMathOperator{\ad}{\textup{ad}}

\DeclareMathOperator{\Exp}{\textup{Exp}}

\DeclareMathOperator{\ch}{\textup{ch}}

\DeclareMathOperator{\Log}{\textup{Log}}

\DeclareMathOperator{\imm}{\textup{im}}
\DeclareMathOperator{\ree}{\textup{re}}

\newtheorem{theorem}{Theorem}[section]

\newtheorem{proposition}[theorem]{Proposition}

\newtheorem{conj}[theorem]{Conjecture}

\theoremstyle{definition}

\newtheorem{example}[theorem]{Example}

\theoremstyle{remark}
\newtheorem{remark}[theorem]{Remark}

\numberwithin{equation}{section}

\begin{document}

\title{Asymptotic behaviour of Kac polynomials}

\author{Lucien Hennecart}

\address{Universit\'e Paris-Saclay, CNRS,  Laboratoire de math\'ematiques d'Orsay, 91405, Orsay, France}

\email{lucien.hennecart@universite-paris-saclay.fr}

\date{\today}

\begin{abstract}
We conjecture a formula supported by computations for the valuation of Kac polynomials of a quiver, which only depends on the number of loops at each vertex. We prove a convergence property of renormalized Kac polynomials of quivers when increasing the number of arrows: they converge in the ring of power series, with a linear rate of convergence. Then, we propose a conjecture concerning the global behaviour of the coefficients of Kac polynomials. All computations were made using SageMath.
\end{abstract}

\maketitle

\section{Introduction}
Kac polynomials are a family of polynomials associated to a quiver, one for each positive root. They are the counting polynomials of absolutely indecomposable representations of fixed dimension vector of the given quiver introduced by Kac in the early eighties (\cite{MR607162,MR677715,MR718127}). These polynomials are of considerable current interest due to their appearance in the different constructions of Yangians associated to a general Kac-Moody algebra and when counting the points of meaningful algebraic stacks, in particular the points of the stack of (nilpotent) representations of the preprojective algebra of the quiver  (\cite{2012arXiv1211.1287M,Okounkovconj,2013arXiv1311.7172D,2016arXiv160102479D,2020arXiv200703289D,BSV,MR4069884}). Lots of conjectures have been proposed by Kac concerning the integrality and positivity of their coefficients and a Lie theoretic interpretation of the constant coefficient (\cite{MR718127}). The integrality and the degree is known since the beginning and the positivity was proved in the last decade by Hausel, Letellier and Rodriguez-Villegas using the geometry of Nakajima quiver varieties (\cite{MR3034296}). Then, Davison interpreted Kac polynomials as Donaldson-Thomas invariants of the tripled quiver associated to $Q$, giving a new proof of the positivity of the coefficients by a purity argument (\cite{MR3826830}). Kac polynomials are now conjectured to give the graded character of the Maulik-Okounkov Lie algebra associated to $Q$ (\cite{Okounkovconj}).

In this paper, we study asymptotic properties of Kac polynomials when making the number of arrows vary and go to infinity. First, we conjecture an explicit formula for the valuation of the polynomials depending only on the number of loops at each vertex. This is supported by computations. Then, we prove that as the number of arrows goes to infinity, renormalized Kac polynomials converge to well-defined power series in $\N[[q]]$ (when the quiver has loops, we need to choose a direction, but we conjecture that the convergence holds without this choice). In fact, we prove a stronger theorem giving a canonical way of writing Kac polynomials as a quotient of two polynomials. The denominator is independent of the number of arrows and its roots are roots of unity. This theorem allows explicit computations of Kac polynomials in a fixed dimension vector as a function of the multiplicity of the arrows. We give several examples of the formulas produced. The code used for the computations is given in the appendix of this paper and is also available on the author's webpage.

\section{Kac polynomials}

\subsection{Quiver representations}
\label{quivrep}
Let $Q=(I,\Omega)$ be a quiver with set of vertices $I$ and set of arrows $\Omega$, both finite. We work over a finite field $\F_q$. A \emph{representation} of $Q$ is the data of a finite dimensional vector space for any vertex and a linear map for any arrow. Finite dimensional representations of $Q$ form an abelian category which has been studied for some time now, beginning with Gabriel (\cite{MR332887}). To $Q$, we associate a bilinear form (its \emph{Euler form}):
\[
 \begin{matrix}
  \langle-,-\rangle&:&\Z^I\times\Z^I&\rightarrow& \Z\\
  &&(\dd,\ee)&\mapsto&\sum_{i\in I}\dd_i\ee_i-\sum_{\alpha:i\rightarrow j}\dd_i\ee_j.
 \end{matrix}
\]
It coincides with the Euler form of the abelian category of finite dimensional representations of $Q$ (which is of finite global dimension), which is by definition the alternate sum of the dimensions of the vector spaces over $\F_q$ given by $\Ext^i_Q(M,N)$ for $M$ of dimension vector $\dd$ and $N$ of dimension vector $\ee$ (and does not depend on the choice of $M,N$).

\subsection{Kac-Moody Lie algebras and Weyl group}
To an arbitrary loop-free quiver $Q=(I,\Omega)$, we can associate a Kac-Moody Lie algebra, given by generators and relations, generalizing the relation between simple Lie algebras and Dynkin graphs. We denote by $\mathfrak{g}_Q$ the Kac-Moody algebra corresponding to $Q$. It has a $\Z^I$-grading. We write $\mathfrak{g}_Q=\bigoplus_{\dd\in\Z^I}\mathfrak{g}_Q[\dd]$. We also associate to $Q$ a Coxeter group $W$, called the Weyl group of the quiver. We refer to \cite{MR1104219} for details on these constructions. A slightly larger class of Lie algebras has been defined by Borcherds using quivers with possible loops in \cite{MR943273}. Although they are very closely related to this work, they won't appear directly here.

\subsection{Counting representations}
Recall that representations of $Q$ over $\F_q$ together with morphisms between them form an abelian category. A representation of $Q$ is called \emph{indecomposable} provided it cannot be non-trivially written as a direct sum of two representations of $Q$. A representation of $Q$ is called \emph{absolutely indecomposable} when it remains indecomposable when tensored over $\F_q$ with the algebraic closure $\overline{\F}_q$ of $\F_q$.
At the begining of the $80$'s, Kac introduced many families of counting functions which happen to be polynomials in $q$ (\cite{MR677715,MR557581,MR607162,MR718127}). They are the number of isomorphism classes of representations of $Q$ over $\F_q$ of dimension vector $\dd$, $M_{Q,\dd}(q)$; the number of isomorphism classes of indecomposable representations of $Q$ over $\F_q$ of dimension vector $\dd$, $I_{Q,\dd}(q)$; and the number of isomorphism classes of absolutely indecomposable representations of $Q$ over $\F_q$ of dimension vector $\dd$, $A_{Q,\dd}(q)$. The family of polynomials $(A_{Q,\dd}(q))_{\dd\in\N^I}$ is the better behaved one, as illustrated by the Kac conjectures proved by Hausel \cite{MR2651380} (Theorem \ref{Hausel}) and Hausel--Letellier--Rodriguez-Villegas \cite{MR3034296} (Theorem \ref{HLRV}). The three families of functions $M_{Q,\dd}(q), I_{Q,\dd}(q)$ and $A_{Q,\dd}(q)$ are related as follows (\cite[Lemma 3.1]{MR3968896}\cite{MR1752774}):
\begin{equation}\label{plethystic}
 \sum_{\dd\geq 0}M_{Q,\dd}(q)z^{\dd}=\Exp_{z}\left(\sum_{\dd>0}I_{Q,\dd}(q)z^{\dd}\right)=\Exp_{q,z}\left(\sum_{\dd>0}A_{Q,\dd}(q)z^{\dd}\right),
\end{equation}
where $\Exp_z$ and $\Exp_{z,t}$ denote the plethystic exponentials, see \cite[Section 1.5]{MR3968896}. The first equality follows from  the Krull-Schmidt property of the category of representations of the quiver and the second from Galois descent for quiver representations.

The results are as follows.
\begin{theorem}[Kac,{\cite[\S\S1.10, 1.15]{MR718127}}]\label{Kac}
 The function $A_{Q,\dd}(q)$ is a polynomial in $q$ with coefficients in $\Z$ of degree $1-\langle\dd,\dd\rangle$. It is nonzero if and only if $\dd$ is a positive root of $\mathfrak{g}_Q$. It does not depend on the orientation of $Q$. If $w\in W$ is an element of the Weyl group such that $\dd$ and $w\dd$ are both positive, then $A_{Q,\dd}(q)=A_{Q,w\dd}(q)$.
\end{theorem}
\begin{proof}[A word on the proof]
 The argument of \cite[\S 1.15]{MR718127} does not suffice to prove that the polynomials $A_{Q,\dd}$ have integer coefficients. Instead one can use the argument of \cite[\S 2.3.1]{BSV} which rests on the deep result of Katz in the appendix of \cite{MR2453601}, at the end of page 616. The same argument allows Bozec, Schiffmann and Vasserot to prove that the nilpotent Kac polynomials have integer coefficients (\cite[\S 2.3.1]{BSV}).
\end{proof}

\begin{theorem}[Hausel,\cite{MR2651380}]\label{Hausel}
 If $Q$ is loop-free and $\dd\in\N^I$ is a positive root, then $A_{Q,\dd}(0)=\dim_{\C}\mathfrak{g}_Q[\dd]$.
\end{theorem}

\begin{theorem}[Hausel--Letellier--Rodriguez-Villegas,\cite{MR3034296}]\label{HLRV}
 The polynomial $A_{Q,\dd}(q)$ has nonnegative coefficients.
\end{theorem}
\begin{remark}
\begin{enumerate}
 \item Theorems \ref{Hausel} and \ref{HLRV} were first proved by Crawley-Boevey and Van den Bergh for indivisible dimension vectors $\dd\in\N^I$ in \cite{MR2038196},
 \item The coefficients of Kac polynomials can be interpreted as the dimensions of certain isotypical components for a Weyl group action on the compactly supported cohomology of a Nakajima quiver variety (\cite{MR3034296}).
\end{enumerate}
\end{remark}

We have a closed formula for the generating function of Kac polynomials due to Hua:
\begin{theorem}[Hua, \cite{MR1752774}]\label{formulaHua}
 \[\sum_{\dd\in\N^I\setminus\{0\}}A_{Q,\dd}(q)z^{\dd}=(q-1)\Log_{z,q}\left(\sum_{\pi=(\pi^i)_{i\in I}\in\mathscr{P}^I}\frac{\prod_{i\rightarrow j\in\Omega}q^{\langle\pi^i,\pi^j\rangle}}{\prod_{i\in I}q^{\langle\pi^i,\pi^i\rangle}\prod_k\prod_{j=1}^{m_k(\pi^i)}(1-q^{-j})}z^{\lvert\pi\rvert}\right),
  \]
where $\Log_{z,q}$ is the plethystic logarithm (\cite[Section 1.5]{MR3968896}), $\mathscr{P}$ is the set of partitions (including the unique partition of $0$), $m_j(\lambda)$ is the multiplicity of the part $j$ in the partition $\lambda$, $\lvert\pi\rvert=(\lvert\pi^i\rvert)_{i\in I}$ and $\langle-,-\rangle$ is the pairing on the set of partitions given by
\[
 \langle\lambda,\mu\rangle=\sum_{i,j}\min(i,j)m_i(\lambda)m_j(\mu).
\]

\end{theorem}

\subsection{Nilpotent versions}
The category of representations of $Q$ has two natural Serre subcategories (stable under extensions, subobjects and quotients): the category of nilpotent representations (the composition of maps along any sufficiently long path is zero), with associated families of counting functions $M_{Q,\dd}^0(q), I_{Q,\dd}^0(q)$, $A_{Q,\dd}^0(q)$ and the category of $1$-nilpotent representations, that is the category of representations of $Q$ so that the restriction of the representation to a vertex $i\in I$ with $g_i$ loops (we only keep the vector space at vertex $i$ and loops at $i$) gives a nilpotent representation of $S_{g_i}$, the quiver with $g_i$ loops, with associated families of counting functions $M_{Q,\dd}^1(q), I_{Q,\dd}^1(q), A_{Q,\dd}^1(q)$. The equalities \eqref{plethystic} are still true for these new families of functions, whose definitions first appear in \cite{BSV}. In \emph{loc. cit.}, Proposition $1.2$, it is proved that these functions are polynomials in $q$, and moreover that $A_{Q,\dd}^{\flat}(q)$ has integer coefficients for $\flat =0,1$. It was conjectured there (Remark $1.3$ of \emph{op. cit.}) and proved by Ben Davison in \cite[Theorem 7.8]{BenPositivity} that these polynomials have nonnegative coefficients. It is important to note that $A^0_{Q,\dd}(q)$ depends on the orientation of $Q$ while $A_{Q,\dd}^{1}(q)$ does not (\cite[Remark 2.8]{BSV}).

\section{Stabilization property of Kac polynomials}
\subsection{Multi-arrows quivers}
Let $Q=(I,\Omega)$ be a quiver. We write $I=I^{\imm}\sqcup I^{\ree}$, where $I^{\imm}$ is the set of vertices of $Q$ having at least one loop (the imaginary vertices) and $I^{\ree}$ its complement (the real vertices). For what follows, we may assume that it has no multiple arrows. Let $\underline{n}=(n_{\alpha})_{\alpha\in\Omega}\in\N^{\Omega}$. The \emph{multi-arrowed quiver} associated to $Q$ and $\underline{n}$ is the quiver $Q_{\underline{n}}=(I,\Omega_{\underline{n}})$ having the same set of vertices as $Q$ but each arrow $\alpha:i\rightarrow j\in\Omega$ of $Q$ is replaced by $n_{\alpha}$ arrows $\alpha_{l}:i\rightarrow j$, $1\leq l\leq n_{\alpha}$. We give two examples in Section \ref{examples}. For $\dd\in\N^{I}$, we consider the Kac polynomials $A_{Q_{\underline{n}},\dd}(q)\in\N[q]$. We are in the situation where $\dd$ is fixed but $\underline{n}$ will vary, and possibly diverge to infinity.

\subsection{Conjectures}
We make the following conjecture, supported by computations for multiloop quivers $S_g$ and tennis-racket quivers (Section \ref{examples}). Given a nonzero polynomial $P\in Q[q]$, we call \emph{valuation} of $P$ the smallest integer $v$ such that the coefficient of $q^v$ in $P$ is nonzero.
\begin{conj}
\label{conjecture}
For $\underline{n}\in(\N_{\geq 1})^{\Omega}$, the valuation of $A_{Q_{\underline{n}},\dd}(q)$ is $v_{\underline{n}}:=\sum_{i\in I^{\imm}}(1+d_i(\sum_{\alpha:i\rightarrow i}n_{\alpha}-1))$ provided the polynomial is nonzero.
\end{conj}
This conjecture is clearly true for loop-free quivers as a consequence of Theorem \ref{Hausel}: if $A_{Q_{\underline{n}},\dd}(q)$ is non-zero, then by Theorem \ref{Kac}, $\dd$ is a positive root of $\mathfrak{g}_{Q_{\underline{n}}}$ hence $A_{Q_{\underline{n}},\dd}(0)=\dim_{\C}(\mathfrak{g}_{Q_{\underline{n}}}[\dd])\neq 0$.

Our interest in this conjecture is the following theorem.

\begin{theorem}
\label{thmconv}
Let $\underline{m}=(m_{\alpha})_{\alpha\in\Omega}\in(\N_{\geq 1}\cup\{+\infty\})^{\Omega}$.
\begin{enumerate}
\item Let $Q$ be a loop-free quiver and $\dd\in\N^I$. Then, for any $\underline{n}\in\N^{\Omega}$ the valuation of $A_{Q_{\underline{n}},\dd}$ is $v_{\underline{n}}=0$ if $A_{Q_{\underline{n}},\dd}\neq 0$. The sequence of polynomials
  \[
   A_{Q_{\underline{n}},\dd}(q)\in\N[q]
  \]
converges in $\N[[q]]$ as $\underline{n}\rightarrow \underline{m}$. Moreover, the limit is the power series expansion at $q=0$ of a rational fraction.

\item Let $Q$ be an arbitrary quiver and $\underline{r}\in\N_{\geq 1}^{\Omega}\setminus\{0\}$. For $s\in \N$,  we let $v_{s\underline{r}}$ be the valuation of the polynomial $A_{Q_{s\underline{r}},\dd}(q)$. Then, the sequence of polynomials
\[
 \frac{A_{Q_{s\underline{r}},\dd}(q)}{q^{v_{s\underline{r}}}}\in\N[q]
\]
converges in $\N[[q]]$ as $s\rightarrow +\infty$ and the limit is the power series expansion at $q=0$ of a rational fraction.
\end{enumerate}

\end{theorem}
In fact, we can prove the following stronger result concerning the structure of Kac polynomials.
\begin{theorem}
\label{thmtech}
 Let $\dd\in \N^I$. Then there exists $r\geq 1$, polynomials with integer coefficients $Q,P_0,\hdots,P_r$ such that the complex roots of $Q$ are roots of unity and affine linear functions $l_0,l_1,\hdots,l_r:\Z^{\Omega}\rightarrow \Z$ whose linear parts $\tilde{l}_j$ have nonnegative coefficients when written in the canonical basis of $\Z^{\Omega}$ (that is, $\tilde{l}_j(e_{\alpha})\geq 0$ for any $0\leq j\leq r$ and $\alpha\in \Omega$) and are pairwise distinct, such that $A_{Q_{\underline{n}},\dd}(q)$ can be written as the rational fraction
 \[
  A_{Q_{\underline{n}},\dd}(q)=\frac{\sum_{j=0}^rq^{l_j(\underline{n})}P_j(q)}{Q(q)}
 \]
for any $\underline{n}\in\N^{\Omega}$. Moreover, if we impose $\gcd(Q,P_0,\hdots,P_r)=1$, $Q$ monic and $P_j(0)\neq 0$ for any $0\leq j\leq r$, the polynomials $Q,P_j$ and the $l_j$ are unique up to reordering.
\end{theorem}
\begin{remark}
\begin{enumerate}
 \item The reader should be tempted to speculate on the expression given by Theorem \ref{thmtech}: is there any (maybe geometric, see Section \ref{geominterp}) interpretation of the affine linear functions $l_j$ or of the polynomials $P_j$? The appearance of $q^{l_j(\underline{n})}$ could possibly be a clue for a hidden affine fibration of rank $l_j(\underline{n})$. At the moment, the author does not have any reasonable conjecture.
 \item Knowing that Kac polynomials can be written under the form given by Theorem \ref{thmtech} is very powerful to give closed formulas for them using computations, see Section \ref{examples}.
 \item The decomposition of Kac polynomials given by this theorem seems related to the refined Kac polynomials defined by Rodriguez-Villegas in \cite{2011arXiv1102.5308R}. However, the formula of Theorem \ref{thmtech} is really a sum of rational fractions and not of polynomials.
 \end{enumerate}
\end{remark}

\begin{proof}[Proof of Theorem \ref{thmconv}]
 We assume Theorem \ref{thmtech}, whose proof will be given below. We prove the first point. Let $\dd\in\N^I$. For $\underline{n}\in\N^{\Omega}$, we write $A_{Q_{\underline{n}},\dd}(q)$ as in Theorem \ref{thmtech}. By the remark preceding Theorem \ref{thmconv}, since $Q$ has no loops, $v_{\underline{n}}=0$ for any $\underline{n}\in\N_{\geq 1}^{\Omega}$ (if $A_{Q_{\underline{n}},\dd}(q)\neq 0$). Therefore, one of the $l_j$'s is the zero affine function $\Z^{\Omega}\rightarrow \Z$. Assume that $l_0=0$. Let $\underline{m}\in(\N_{\geq 1}\cup\{+\infty\})^{\Omega}$. Let $J'=\{1\leq j\leq r\mid \lim_{\underline{n}\rightarrow\underline{m}}\limits l_j(\underline{n})=+\infty\}$ and $J=\{1,\hdots,r\}\setminus J'$. For $j\in J$, $l_j(\underline{n})$ is a sequence of integers which stabilizes as $\underline{n}\rightarrow \underline{m}$ (by the positivity assumption on the coefficients of the linear forms $l_j$, $0\leq j\leq r$). We let $l_{j}(\underline{m})$ be its limit. In particular, $0\in J$ and $l_0(\underline{m})=0$. Then,
 \[
  \lim_{\underline{n}\rightarrow \underline{m}}\limits \frac{A_{Q_{\underline{n}},\dd}(q)}{q^{v_{\underline{n}}}}=\frac{\sum_{j\in J}q^{l_j(\underline{m})}P_j(q)}{Q(q)}.
 \]
In particular, for $\underline{m}=(+\infty,\hdots,+\infty)$, $J=\{0\}$ and so
\[
  \lim_{\underline{n}\rightarrow \underline{m}}\limits \frac{A_{Q_{\underline{n}},\dd}(q)}{q^{v_{\underline{n}}}}=\frac{P_0(q)}{Q(q)}.
 \]

For the second point, let $Q$ be an arbitrary quiver and $\underline{r}\in\N_{\geq 1}^{\Omega}$. Write $A_{Q_{\underline{n}},\dd}(q)$ as in Theorem \ref{thmtech}. We obtain affine functions
\[
 \begin{matrix}
  l'_j &:&\Z&\rightarrow& \Z\\
  &&s&\mapsto&l_j(s\underline{r})
 \end{matrix}
\]
for $0\leq j\leq r$. Let $J=\{1\leq j\leq r\mid l'_j \text{ is constant}\}$. For $j\in J$, we let $l'_j(\infty)$ be the unique value taken by $l'_j$. Then,
\[
 \lim_{s\rightarrow +\infty}\frac{A_{Q_{s\underline{r}},\dd}}{q^{v_{s\underline{r}}}}=\frac{\sum_{j\in J}q^{l'_j(\infty)}P_j(q)}{Q(q)}.
\]
This proves the second point of the theorem.
\end{proof}
In the second point of Theorem \ref{thmconv}, we need to choose a direction $\underline{r}$ since it is not clear that the set $J$ in the proof is independent of the direction (and the limit clearly depends on $J$). Nethertheless, we conjecture it is the case. If this conjecture is true, the limit is independent of the given direction $\underline{r}$. This conjecture is therefore equivalent to the following one.
\begin{conj}
 Let $Q$ be a quiver, $\dd\in\N^I$ and $v_{\underline{n}}$ be the valuation of $A_{Q_{\underline{n}},\dd}(q)$. Then, the sequence of polynomials
 \[
  \frac{A_{Q_{\underline{n},\dd}}(q)}{q^{v_{\underline{n}}}}\in\N[q]
 \]
converges in $\N[[q]]$ as $\underline{n}\rightarrow (+\infty,\hdots,+\infty)$ and the limit is the power series expansion at $0$ of a rational fraction.
\end{conj}

\begin{proof}[Proof of Theorem \ref{thmtech}]
 We take advantage of the explicit formula of Theorem \ref{formulaHua}. First recall the formula for the plethystic logarithm (\cite[Section 1.5]{MR3968896}):
 \[
 \Log_{q,z}:1+\Q[[q,z_i:i\in I]]^+\rightarrow \Q[[q,z_i:i\in I]]
\]
 where the $+$ upperscript denotes the augmentation ideal, the ideal of $\Q[[q,z_i:i\in I]]$ generated by $q,z_i, i\in I$. For $f\in \Q[[q,z_i:i\in I]]^+$,
 \[
  \Log_{q,z}(1+f)=\sum_{l\geq 1}\frac{\mu(l)}{l}\psi_l\left(\sum_{k\geq 1}\frac{(-1)^{k+1}f^k}{k}\right),
 \]
where $\mu$ is Moebius function, and for $g=g(q,z_i:i\in I)\in \Q[[q,z_i,i\in I]]$, $\psi_l(g)=g(q^l,z_i^l:i\in I)$. For $\underline{n}\in\N^{\Omega}$, we let
\[
 f_{\underline{n}}=\sum_{\pi=(\pi^i)_{i\in I}\in\mathscr{P}'^I}\frac{\prod_{i\xrightarrow{\alpha} j\in\Omega}q^{n_{\alpha}\langle\pi^i,\pi^j\rangle}}{\prod_{i\in I}q^{\langle\pi^i,\pi^i\rangle}\prod_k\prod_{j=1}^{m_k(\pi^i)}(1-q^{-j})}z^{\lvert\pi\rvert}
\]
where $\mathscr{P}'$ is the set of partitions of integers not including the trivial partition $\emptyset$ of $0$. Then, by Theorem \ref{formulaHua}, we have
\[
 \sum_{\dd\in\N^I}A_{Q_{\underline{n}},\dd}(q)z^{\dd}=(q-1)\Log(1+f_{\underline{n}}).
\]
From this, we see that for fixed dimension vector $\dd\in\N^I$, the Kac polynomial $A_{Q_{\underline{n}},\dd}(q)$ can be written as a finite sum of rational fractions whose denominators are products of terms of the form $q^a(1-q^{-b})$ and whose numerators are finite linear combinations of powers of $q$ where the exponents are degree one polynomials in the $n_{\alpha}$, $\alpha\in\Omega$. This proves that one can write Kac polynomials as in Theorem \ref{thmtech}. Since for any partitions $\pi,\pi'$, the pairing $\langle\pi,\pi'\rangle$ is a nonnegative integer, this implies the nonnegativity of the coefficients of the linear part of the affine linear functions $l_j$, $0\leq j\leq r$.

For the unicity, assume that for any $\underline{n}\in\N^I$ one has
\[
 \frac{\sum_{j=0}^rq^{l_j(\underline{n})}P_j(q)}{Q(q)}=\frac{\sum_{j=0}^sq^{l'_j(\underline{n})}P'_j(q)}{Q'(q)}
\]
for some affine linear functions $l_0,\hdots,l_r,l'_0,\hdots,l'_s$ and polynomials $P_0,\hdots,P_r,P'_0,\hdots,P'_s$ verifying the assumptions given in Theorem \ref{thmtech}. Then,
\[
 \sum_{j=0}^rq^{l_j(\underline{n})}P_j(q)Q'(q)=\sum_{j=0}^sq^{l'_j(\underline{n})}P'_j(q)Q(q).
\]
By comparing the behaviour of the degrees as $\underline{n}$ varies, we see that $r=s$, the sets $\{l_0,\hdots,l_r\}$ and $\{l'_0,\hdots,l'_s\}$ coincide and for $j,j'$ such that $l_j=l_{j'}$
\[
 P_j(q)Q'(q)=P'_{j'}(q)Q(q).
\]
Up to reordering, we can assume that $l_j=l'_j$ for $0\leq j\leq r$. Since $\gcd(P_0,\hdots,P_r,Q)=1$, this implies that $Q$ divides $Q'$ and conversely that $Q'$ divides $Q$. Therefore, $Q=Q'$ and $P_j=P'_j$.
\end{proof}

By Theorem \ref{thmconv}, to a set of vertices $I$ we can associate a family of power series $A_{I,\dd}(q)\in\N[[q]]$ as follows. Let $Q=(I,\Omega)$ be a quiver given by an arbitrary orientation of the complete graph on the set of vertices $I$. We let $\underline{\infty}=(+\infty,\hdots,+\infty)\in(\N\cup\{+\infty\})^{\Omega}$ and for $\dd\in\N^I$,
\[
 A_{I,\dd}(q)=\lim_{\underline{n}\rightarrow\underline{\infty}}A_{Q_{\underline{n}},\dd}(q).
\]
As Kac polynomials do not depend on the orientation of $Q$, this family of power series only depends on the cardinality of $I$.  For such quivers, the support of any dimension vector is connected (since any two vertices are connected). Therefore, an easy computation (one checks that $\langle\dd,\dd\rangle\leq 0$ for $\underline{n}$ big enough where $\langle-,-\rangle$ denotes the Euler form defined in Section \ref{quivrep}) shows that any $\dd\in\N^{I}\setminus \{0,e_i, i\in I\}$ (where $e_i$, $i\in I$ denotes the canonical basis of $\N^I$ as a monoid) is an imaginary root of $Q_{\underline{n}}$ for $\underline{n}$ big enough (see \cite[Lemma 5.3]{MR1104219}). So $A_{Q_{\underline{n}},\dd}(q)\neq 0$ for $\underline{n}\gg 0$. By letting $\underline{n}\rightarrow \underline{\infty}$, we consequently get a non-zero power series $A_{Q,\dd}(q)$ for any $\dd\in\N^I$. It is also possible to imagine lots of variants using for example the same process for a quiver having exactly one arrow between any two vertices and one loop at each vertex.

\begin{remark}
 Conjecture \ref{conjecture} can also be formulated for the polynomials $A_{Q_{\underline{n}},\dd}^{\flat}(q)$ for $\flat =0,1$. We conjecture that the valuation of $A_{Q_{\underline{n}},\dd}^{\flat}(q)$ is $0$.
\end{remark}

\begin{theorem}
\label{recipr}
 Let $\underline{m}=(m_{\alpha})_{\alpha\in\Omega}\in(\N_{\geq 1}\cup\{+\infty\})^{\Omega}$. The sequence of polynomials
 \[
  q^{\deg A_{Q_{\underline{n}},\dd}(q)}A_{Q_{\underline{n}},\dd}(q^{-1})
 \]
converges in $\N[[q]]$ as $\underline{n}\rightarrow \underline{m}$.
\end{theorem}
By a result of Kac (\cite[\S 1.15]{MR718127}), the polynomials $A_{Q_{\underline{n}},\dd}$ are all monic. Hence, the limit of Theorem \ref{recipr} is an element of $1+q\N[[q]]$.

\begin{proof}
 For $\underline{n}\in\N^{\Omega}$, we write $A_{Q_{\underline{n}},\dd}(q)=\frac{\sum_{j=0}^rq^{l_j(\underline{n})}P_j(q)}{Q(q)}$ as in Theorem \ref{thmtech}, where the conditions on the polynomials and affine functions appearing in this decomposition are satisfied, which ensures the unicity of the decomposition. Since we are interested in the limit $\underline{n}\rightarrow\underline{m}$, we can restrict ourselves to consider $\underline{n}$ in the subset
 \[
  \N^{\Omega}_{\underline{m}}=\{\underline{n}\in\N^{\Omega}\mid \text{for any $\alpha\in\Omega$, }m_{\alpha}<+\infty\implies n_{\alpha}=m_{\alpha}\}.
 \]
 
 We let $\Omega'=\{\alpha\in\Omega\mid m_{\alpha}=+\infty\}$. We identify $\N^{\Omega'}$ and $\N^{\Omega}_{\underline{m}}$ by sending $(n_{\alpha}')_{\alpha\in\Omega'}$ to $(n_{\alpha})_{\alpha\in\Omega}$, where $n_{\alpha}=n_{\alpha}'$ if $\alpha\in\Omega'$, $n_{\alpha}=m_{\alpha}$ else. For $\alpha\in\Omega'$, we denote by $e_{\alpha}\in\N^{\Omega}_{\underline{n}}$ the vector such that $(e_{\alpha})_{\beta}=m_{\beta}$ if $\beta\not\in\Omega'$, $(e_{\alpha})_{\alpha}=1$ and $(e_{\alpha})_{\beta}=0$ if $\beta\in\Omega'\setminus\{\alpha\}$. The previous identification endows $\N^{\Omega}_{\underline{m}}$ with a monoid structure: for $\underline{n},\underline{n'}\in\N^{\Omega}_{\underline{m}}$, 
 \[
  (\underline{n}+\underline{n'})_{\alpha}=
  \left\{
  \begin{aligned}
   &n_{\alpha}+n'_{\alpha}&\text{ if $\alpha\in\Omega'$}\\
   &m_{\alpha}&\text{ else}
  \end{aligned}
  \right.
 \]
 We define in a similar way $\Z^{\Omega}_{\underline{m}}$ together with a $\Z$-module structure on it. We consider the restrictions $l'_j$ ($0\leq j\leq r$) of the affine functions $l_j$ to $\Z^{\Omega}_{\underline{m}}$. By sorting the $l'_j$'s according to their linear parts, we obtain the existence of an integer $s\geq 0$, affine functions $\tilde{l}_j : \Z^{\Omega}_{\underline{m}}\rightarrow\Z$, $0\leq j\leq s$, whose linear parts are pairwise distinct and have nonnegative coefficients when written in the canonical basis $(e_{\alpha})_{\alpha\in\Omega'}$ of $\Z^{\Omega}_{\underline{m}}$, polynomials with integer coefficients $\tilde{Q},\tilde{P_0},\hdots,\tilde{P_s}$ such that the complex roots of $\tilde{Q}$ are roots of unity, $\tilde{Q}$ is monic, $\tilde{P_j}(0)\neq 0$ for any $0\leq j\leq s$ and $\gcd( \tilde{Q},\tilde{P_1},\hdots,\tilde{P_s})=1$, such that
 \begin{equation}
 \label{decprime}
  A_{Q_{\underline{n}},\dd}(q)=\frac{\sum_{j=0}^sq^{\tilde{l}_j(\underline{n})}\tilde{P_j}(q)}{\tilde{Q}(q)}
 \end{equation}
for any $\underline{n}\in\N^{\Omega}_{\underline{m}}$. Moreover, this decomposition is unique by the arguments of the proof of Theorem \ref{thmtech}. By Theorem \ref{Kac}, for $\underline{n}\in\Z^{\Omega}$, the degree of the polynomial $A_{Q_{\underline{n}},\dd}(q)$ is given by the affine function
\[
\begin{matrix}
 \deg &:& \Z^{\Omega}&\rightarrow&\Z\\
 &&(\underline{n}_{\alpha})_{\alpha\in\Omega}&\mapsto&1-\langle\dd,\dd\rangle_{\underline{n}}
\end{matrix}
\]
where $\langle-,-\rangle_{\underline{n}}$ denotes the Euler form of the quiver $Q_{\underline{n}}$ (Section \ref{quivrep}). More explicitly,
\[
 \deg\underline{n}=1-\sum_{i\in I}\dd_i^2+\sum_{i\xrightarrow{\alpha}j\in\Omega}n_{\alpha}\dd_i\dd_j
\]
for any $\underline{n}=(n_{\alpha})_{\alpha\in\Omega}\in\Z^{\Omega}$. We consider the restriction of $\deg$ to $\Z^{\Omega}_{\underline{m}}$ and denote it $\widetilde{\deg}$. For any affine function $l$ on some space, we let $L(l)$ denote its linear part. As a consequence of the degree formula and the unicity of the decomposition \eqref{decprime}, there is a unique $0\leq j\leq s$ such that $\tilde{l}_j$ has the same linear part as $\widetilde{\deg}$. We assume this is $l_0$. Then, we claim that for any $1\leq j\leq s$ and for any $\underline{n}\in\N^{\Omega}_{\underline{m}}$,
\[
 L(\tilde{l}_j)(\underline{n})\leq L(\tilde{l}_0)(\underline{n})
\]
and there exists $\alpha\in\Omega'$ such that for $\underline{n}=e_{\alpha}\in \N^{\Omega}_{\underline{m}}$, the inequality is strict. We first note that by linearity, the claim is equivalent to the same inequality for any $1\leq j\leq s$ and any $\underline{n}\in(\Q_+)^{\Omega}_{\underline{m}}$ and by density, we can replace $\Q_+$ by $\R_+$. Assume it is false. Then, there exists $1\leq j\leq s$ and $\underline{n}\in(\Q_+)^{\Omega}_{\underline{m}}$ such that $L(\tilde{l}_j)(\underline{n})>L(\tilde{l}_0)(\underline{n})$. The same inequality remains true in a small neighbourhood $U$ of $\underline{n}$ in $(\R_+)^{\Omega}_{\underline{m}}$. Since the $\tilde{l}_j$'s  have pairwise distinct linear parts,
\[
 V:=U\setminus \bigcup_{u\neq v}\{\underline{n}\in(\R_+)^{\Omega}_{\underline{m}}\mid L(\tilde{l}_u)(\underline{n})=L(\tilde{l}_v)(\underline{n}) \}
\]
is a non-empty open subset of $U$. For each connected component $V_0$ of $V$, there exists $1\leq j'\leq s$ such that for any $0\leq u\leq s$, $u\neq j'$ and $\underline{n}\in V_0$,
\[
 L(\tilde{l}_{j'})(\underline{n})>L(\tilde{l}_u)(\underline{n}).
\]
Fix a connected component $V_0$ of $V$ and $j'$ as above. Let $\underline{n}\in V_0\cap (\Q_+)^{\Omega}_{\underline{m}}$ and $a\in\N_{\geq 1}$ be such that $a\underline{n}\in\N^{\Omega}_{\underline{m}}$. Then, for any $b\in\N$, $ba\underline{n}\in\N^{\Omega}_{\underline{m}}$ and
\[
 L(\tilde{l}_{j'})(ab\underline{n})-L(\tilde{l}_u)(ab\underline{n})\rightarrow+\infty
\]
as $b\rightarrow +\infty$ for any $0\leq u\leq s$, $u\neq j'$. Therefore, for $b\gg 0$,
\[
 \deg A_{Q_{ab\underline{n}},\dd}(q)=\tilde{l}_{j'}(ab\underline{n})+\deg\tilde{P_{j'}}-\deg\tilde{Q}>\tilde{l}_{0}(ab\underline{n})+\deg\tilde{P_{0}}-\deg\tilde{Q}
\]
and this is a contradiction with the degree formula for $\deg A_{Q_{ab\underline{n}},\dd}(q)$. Now, since $L(\tilde{l}_j)$ and $L(\tilde{l}_0)$ are distinct linear functions on $\Z^{\Omega}_{\underline{m}}$, there exists $\alpha\in\Omega'$ such that $L(\tilde{l}_j)(e_{\alpha})<L(\tilde{l}_0)(e_{\alpha})$. Consequently, as $\underline{n}\rightarrow \underline{m}$, $\underline{n}\in\N^{\Omega}_{\underline{m}}$, $L(\tilde{l}_0)(\underline{n})-L(\tilde{l}_j)(\underline{n})\rightarrow +\infty$ for $1\leq j\leq s$. Since $\widetilde{\deg}$ and $\tilde{l}_0$ have the same linear part, as $\underline{n}\rightarrow \underline{m}$, $\underline{n}\in\N^{\Omega}_{\underline{m}}$,
\[
 \widetilde{\deg}(\underline{n})-\tilde{l}_j(\underline{n})\rightarrow+\infty
\]
and $ \widetilde{\deg}-\tilde{l}_0:=c$ is constant.
Then, for $\underline{n}\in\N^{\Omega}_{\underline{m}}$, we have
\[
 q^{\deg A_{Q_{\underline{n}},\dd}(q)}A_{Q_{\underline{n}}\dd}(q^{-1})=\frac{q^{\widetilde{\deg}(\underline{n})-\tilde{l}_0(\underline{n})}\tilde{P_0}(q^{-1})+\sum_{j=1}^sq^{\widetilde{\deg}(\underline{n})-\tilde{l}_j(\underline{n})}\tilde{P_j}(q^{-1})}{\tilde{Q}(q^{-1})}.
\]
The limit as $\underline{n}\rightarrow\underline{m}$ is then
\[
 \lim_{\underline{n}\rightarrow\underline{m}}\limits q^{\deg A_{Q_{\underline{n}},\dd}(q)}A_{Q_{\underline{n}}\dd}(q^{-1})=\frac{q^c\tilde{P_0}(q^{-1})}{\tilde{Q}(q^{-1})}.
\]
This proves the theorem.

\end{proof}

\subsection{Particular cases}
\label{examples}
To support Conjecture \ref{conjecture}, we give some examples for the quiver $S_g$ with one vertex and $g$-loops (multi-loop quivers) and for the tennis-racket quiver (Section \ref{arrowloop}):
\[
S_g=\begin{tikzcd}
1 \arrow[out=-30,in=30,loop,swap,"(g)"]
\end{tikzcd}
\]

and the generalized Kronecker quivers $K_r$ with two vertices and $r$ arrows from the first to the second:
\[
 K_r=\begin{tikzcd}
  1\arrow[r,"(r)"]& 2
 \end{tikzcd}
\]
If $Q=S_g$ and $\underline{n}=(n_1,\hdots,n_g)$ (we label the arrows of $Q$ by the integers $1,\hdots, g$), then $Q_{\underline{n}}$ is the one-vertex quiver with $g\cdot\sum_{i=1}^gn_i$ loops:
\[
 Q_{\underline{n}}=
 \begin{tikzcd}
1 \arrow[out=-30,in=30,loop,swap,"(g\cdot\sum_{i=1}^gn_{i})"]
\end{tikzcd}
\]
and if $Q=K_r$, $\underline{n}=(n_1,\hdots,n_r)$, $Q_{\underline{n}}$ is the quiver having two vertices and $r\cdot\sum_{i=1}^rn_i$ arrows from the first to the second:
\[
 Q_{\underline{n}}=
 \begin{tikzcd}[column sep=15ex]
  1\arrow[r,"\left(r\cdot\sum_{i=1}^rn_i\right)"]& 2
 \end{tikzcd}
\]

\begin{example}\label{exKronecker}
 By Theorem \ref{thmconv}, we let $A_{K_{\infty},\dd}=\lim_{r\rightarrow \infty}\limits A_{K_r,\dd}$. This Section illustrates Theorem \ref{thmconv} and its proof. We have the following equalities:
 \[
  \begin{aligned}
  A_{K_r,(2,2)}&=\frac{q^{4r-2}-q^{2(r-1)}(1+2q+3q^2)+2q^{r-1}(q+1)^2-(2q+1)}{(q-1)^3(q+1)^2}\\
   A_{K_{\infty},(2,2)}&=\frac{1+2q}{(1-q)^3(q+1)^2}\\
   &=1+3q+5q^2+9q^3+12q^4+18q^5+22q^6+30q^7+35q^8+45q^9+51q^{10}+O(q^{11})\\
   A_{K_{r},(2,3)}&=\frac{1}{(1-q)^3(1+q)^2(1-q^3)}\left((q^{6r-4}-q^{4r-4}(1+q+q^2)-q^{3r-2}(1+3q+3q^2+q^3)+\right.&\\
   &\left.q^{2r-3}(1+3q+7q^2+8q^3+6q^4+2q^5)-q^{r-2}(1+4q+7q^7+7q^3+4q^4+q^5)+(2+2q+2q^2+q^3)\right)\\
   A_{K_{\infty},(2,3)}&=\frac{2+2q+2q^2+q^3}{(1-q)^3(1+q)^2(1-q^3)}\\
   &=2+4q+10q^2+17q^3+29q^4+43q^5+64q^6+87q^7+119q^8+154q^9+199q^{10}+O(q^{11})\\
    A_{K_{r},(2,4)}&=\frac{1}{(1-q)^4(1+q)^3(1-q^3)(1+q^2)}\left(-q^{8r-7}
    +q^{6r-7}(1+q+q^2+q^3)\right.\\
    &\left.-q^{4r-5}(1+2q-q^2-q^3-5q^4-4q^5-3q^6)-q^{3r-6}(1+q+q^2+4q^3+8q^4+13q^5+15q^6+\right.\\
    &\left.12q^7+7q^8+2q^9)+q^{2r-5}(1+2q+4q^2+10q^3+15q^4+22q^5+19q^6+17q^7+7q^8+3q^9)
     \right.\\
    &\left.-q^{r-3}(1+3q+7q^2+12q^3+14q^4+13q^5+9q^6+4q^7+q^8)+(2+4q+3q^2+4q^3+q^4+q^5)    
    \right)\\
    A_{K_{\infty},(2,4)}&=\frac{2+4q+3q^2+4q^3+q^4+q^5}{(1-q)^4(1+q)^3(1-q^3)(1+q^2)}\\
    &=2 + 6q + 13q^2 + 27q^3 + 46q^4 + 78q^5 + 118q^6 + 179q^7 + 251q^8 + 355q^9 + 473q^{10}+O(q^{11})\\
    A_{K_{\infty},(2,5)}&=3 + 7q + 18q^2 + 35q^3 + 67q^4 + 113q^5 + 186q^6 + 286q^7 + 431q^8 + 622q^9 + 882q^{10}+O(q^{11})
    \\
   A_{K_{\infty},(3,3)}&=3 + 9q + 24q^2 + 48q^3 + 92q^4 + 154q^5 + 248q^6 + 376q^7 + 551q^8 + 775q^9 + 1070q^{10}+O(q^{11})\\
   A_{K_{\infty},(3,4)}&=5 + 15q + 44q^2 + 98q^3 + 200q^4 + 364q^5 + 631q^6 + 1021q^7 + 1596q^8 + 2390q^9 + 3485q^{10}+O(q^{11})\\
  A_{K_{\infty},(3,5)}&=7 + 23q + 70q^2 + 165q^3 + 355q^4 + 685q^5 + 1247q^6 + 2129q^7 + 3491q^8 + 5488q^9 + 8370q^{10}+O(q^{11})\\
  A_{K_{\infty},(4,4)}&=8 + 32q + 98q^2 + 250q^3 + 547q^4 + 1101q^5 + 2036q^6 + 3574q^7 + 5933q^8 + 9513q^9 + 14658q^{10}+O(q^{11})
  \end{aligned}
 \]
\end{example}

\subsubsection{The multi-loops quivers}\label{multloop}
Let $S_g$ be the quiver with one vertex and $g$ loops and $A_{S_g,d,0}=\frac{A_{S_g,d}}{q^{1+d(g-1)}}$. We let
\[
 A_{S_{\infty},d,0}=\lim_{g\rightarrow\infty}\limits A_{S_g,d,0}.
\]
We have the following:
\[
 \begin{aligned}
  A_{S_{\infty},1,0}&=1\\
  A_{S_{g},2,0}&=\frac{1-q^{2g}}{1-q^2}\\
  A_{S_{\infty},2,0}&=\frac{1}{1-q^2}\\
  &=1 + q^2 + q^4 + q^6 + q^8 + q^{10}+O(q^{11})\\
  A_{S_{g},3,0}&=\frac{(1+q)-q^{2g-1}(1+q+q^2)+q^{6g-1}}{(1-q^2)(1-q^3)}\\
  A_{S_{\infty},3,0}&=\frac{1}{(1-q)(1-q^3)}\\
  &=1 + q + q^2 + 2q^3 + 2q^4 + 2q^5 + 3q^6 + 3q^7 + 3q^8 + 4q^9 + 4q^{10}+O(q^{11})\\
  A_{S_{\infty},4,0}&=\frac{1}{(1-q^2)^2(1-q^3)(1+q^2)}\left((1+q+2q^2+q^3+q^4)-q^{2g-1}(1+2q+3q^2+3q^3+2q^4+q^5)+
  \right.\\
  &q^{4g-2}(-1+2q^3+q^4+q^5)+
  \left.q^{6g-3}(1+q+q^2+q^3)-q^{12g-3}\right)\\
  A_{S_{\infty},4,0}&=\frac{1+q+2q^2+q^3+q^4}{(1-q^2)^2(1-q^3)(1+q^2)}=\frac{1}{(1-q)(1-q^2)^2}\\
  &=1 + q + 3q^2 + 3q^3 + 6q^4 + 6q^5 + 10q^6 + 10q^7 + 15q^8 + 15q^9 + 21q^{10}+O(q^{11})\\
  A_{S_{\infty},5,0}&=\frac{1+2q+3q^2+4q^3+4q^4+4q^5+3q^6+2q^7+q^8}{(1-q)^3(1+q)^2(1-q^3)(1+q^2)(1-q^5)}=\frac{1-q+q^2}{(1-q)^3(1-q^5)}\\
  &=1 + 2q + 4q^2 + 7q^3 + 11q^4 + 17q^5 + 24q^6 + 33q^7 + 44q^8 + 57q^9 + 73q^{10}+O(q^{11})\\
  A_{S_{\infty},6,0}&=\frac{1+q+2q^2+2q^3+3q^4+2q^6+q^7}{(1-q)^5(1+q)^3(1-q+q^2)(1+q+q^2)}\\
  &=1 + 2q + 6q^2 + 11q^3 + 22q^4 + 33q^5 + 57q^6 + 80q^7 + 121q^8 + 164q^9 + 231q^{10}+O(q^{11})\\
  A_{S_{\infty},7,0}&=\frac{(1-q+q^2)^2}{(q-1)^6(q^6+q^5+q^4+q^3+q^2+q+1)}\\
  &=1 + 3q + 8q^2 + 18q^3 + 36q^4 + 66q^5 + 113q^6 + 184q^7 + 286q^8 + 429q^9 + 624q^{10}+O(q^{11})\\
  A_{S_{\infty},8,0}&=1 + 3q + 11q^2 + 25q^3 + 59q^4 + 113q^5 + 217q^6 + 371q^7 + 630q^8 + 994q^9 + 1554q^{10}+O(q^{11})\\
  A_{S_{\infty},9,0}&=1 + 4q + 13q^2 + 37q^3 + 88q^4 + 190q^5 + 382q^6 + 715q^7 + 1270q^8 + 2162q^9 + 3536q^{10}+O(q^{11})\\
  A_{S_{\infty},10,0}&=1 + 4q + 17q^2 + 48q^3 + 130q^4 + 297q^5 + 647q^6 + 1280q^7 + 2438q^8 + 4363q^9 + 7571q^{10}+O(q^{11})
 \end{aligned} 
\]
\subsubsection{The tennis-racket quiver}\label{arrowloop}We consider the quiver
\[
Q=\begin{tikzcd}
1\arrow[r,"\alpha"]&2 \arrow[out=0,in=90,loop,swap,"\beta"]
\end{tikzcd}.
\]
We let $A_{Q_{\underline{n}},\dd,0}=\frac{A_{Q_{\underline{n}},\dd}}{q^{1+d_2(n_{\beta}-1)}}$ and
\[
 A_{Q_{\underline{\infty},\dd,0}}=\lim_{\underline{n}\rightarrow\underline{\infty}}A_{Q_{\underline{n}},\dd}.
\]
We have the following formulas, obtained combining Theorem \ref{thmtech} and computations. We let $\underline{n}=(n_{\alpha},n_{\beta})$.
\[
 \begin{aligned}
  A_{Q_{\underline{n}},(1,1),0}&=\frac{1-q^{n_{\alpha}}}{1-q}\\
  A_{Q_{\underline{\infty}},(1,1),0}&=\frac{1}{1-q}\\
  &=1 + q + q^2 + q^3 + q^4 + q^5 + q^6 + q^7 + q^8 + q^9 + q^{10}+O(q^{11})\\
  A_{Q_{\underline{n}},(1,2),0}&=\frac{1+q-q^{n_{\alpha}}(1+q)-q^{2n_{\beta}}+q^{2(n_{\alpha}+n_{\beta})}}{(1-q)(1-q^2)}\\ 
  A_{Q_{\underline{\infty}},(1,2),0}&=\frac{1+q}{(1-q)(1-q^2)}\\
  &=1 + 2q + 3q^2 + 4q^3 + 5q^4 + 6q^5 + 7q^6 + 8q^7 + 9q^8 + 10q^9 + 11q^{10}+O(q^{11})\\
  A_{Q_{\underline{n}},(1,3),0}&=\frac{1}{(1-q)(1-q^2)(1-q^3)}\left((q+1)(q^2+q+1)-q^{n_{\alpha}}(q+1)(q^2+q+1)-q^{2n_{\beta}-1}(q^2+1)(q^2+q+1)\right.\\
  &\left.(1+q+q^2)(q^{n_{\alpha}+2n_{\beta}+1}+q^{2n_{\alpha}+2n_{\beta}+1})+q^{6n_{\beta}-1}-q^{3n_{\alpha}+6n_{\beta}-1}\right)\\
  A_{Q_{\underline{\infty}},(1,3),0}&=\frac{1}{(1-q)^3}\\
  &=1 + 3q + 6q^2 + 10q^3 + 15q^4 + 21q^5 + 28q^6 + 36q^7 + 45q^8 + 55q^9 + 66q^{10}+O(q^{11})\\
  A_{Q_{\underline{\infty}},(1,4),0}&=1 + 4q + 10q^2 + 20q^3 + 35q^4 + 56q^5 + 84q^6 + 120q^7 + 165q^8 + 220q^9 + 286q^{10}+O(q^{11})\\
  A_{Q_{\underline{n}},(2,1),0}&=\frac{1-q^{n_{\alpha}-1}(1+q)+q^{2n_{\alpha}}}{(1-q)(1-q^2)}\\
  A_{Q_{\underline{\infty}},(2,1),0}&=\frac{1}{(1-q)(1-q^2)}\\
  &=1 + q + 2q^2 + 2q^3 + 3q^4 + 3q^5 + 4q^6 + 4q^7 + 5q^8 + 5q^9 + 6q^{10}+0(q^{11})\\
  A_{Q_{\underline{n}},(2,2),0}&=\frac{1+q+2q^2-q^{n_{\alpha}}(2+4q+2q^2)-q^{2n_{\beta}}+q^{2n_{\alpha}}(1+3q)+q^{2(n_{\alpha}+n_{\beta})}(1+q)-q^{4n_{\alpha}+2n_{\beta}-1}}{(1-q^2)^2(1-q)}\\
  A_{Q_{\underline{\infty}},(2,2),0}&=\frac{1+q+2q^2}{(1-q^2)^2(1-q)}\\
  &=1 + 2q + 6q^2 + 8q^3 + 15q^4 + 18q^5 + 28q^6 + 32q^7 + 45q^8 + 50q^9 + 66q^{10}+O(q^{11})\\
  A_{Q_{\underline{\infty}},(2,3),0}&=\frac{1+q+q^2}{(1-q^2)(1-q)^3}\\
  &=1 + 4q + 11q^2 + 23q^3 + 42q^4 + 69q^5 + 106q^6 + 154q^7 + 215q^8 + 290q^9 + 381q^{10}+O(q^{11})\\
  A_{Q_{\underline{\infty}},(2,4),0}&=1 + 5q + 19q^2 + 45q^3 + 100q^4 + 182q^5 + 322q^6 + 510q^7 + 795q^8 + 1155q^9 + 1661q^{10}+O(q^{11})\\
  A_{Q_{\underline{\infty}},(3,1),0}&=1 + q + 2q^2 + 3q^3 + 4q^4 + 5q^5 + 7q^6 + 8q^7 + 10q^8 + 12q^9 + 14q^{10}+O(q^{11})\\
  A_{Q_{\underline{\infty}},(3,2),0}&=1 + 3q + 7q^2 + 14q^3 + 24q^4 + 38q^5 + 57q^6 + 81q^7 + 111q^8 + 148q^9 + 192q^{10}+O(q^{11})\\
  A_{Q_{\underline{\infty}},(3,3),0}&=\frac{(2q+1)(q^5+q^4+2q^3+3q^2+q+1)}{(1-q^2)(1-q)^2(1-q^3)^2}\\  
  &=1 + 5q + 15q^2 + 38q^3 + 78q^4 + 144q^5 + 248q^6 + 397q^7 + 605q^8 + 890q^9 + 1261q^{10}+O(q^{11})\\
  A_{Q_{\underline{\infty}},(3,4),0}&=1 + 7q + 27q^2 + 79q^3 + 191q^4 + 405q^5 + 779q^6 + 1390q^7 + 2336q^8 + 3740q^9 + 5751q^{10}+O(q^{11})\\
  A_{Q_{\underline{\infty}},(4,1),0}&=1 + q + 2q^2 + 3q^3 + 5q^4 + 6q^5 + 9q^6 + 11q^7 + 15q^8 + 18q^9 + 23q^{10}+O(q^{11})\\
  A_{Q_{\underline{\infty}},(4,2),0}&=1 + 3q + 9q^2 + 17q^3 + 35q^4 + 56q^5 + 95q^6 + 139q^7 + 211q^8 + 290q^9 + 410q^{10}=O(q^{11})
 \end{aligned}
\]

\subsection{The constant coefficient of Kac polynomials}
When $Q$ has no edge-loops, the constant coefficient of $A_{Q_{\underline{n}},\dd}$ can easily be seen to converge when $\underline{n}\rightarrow \underline{m}$. Indeed, for a connected dimension vector and $\underline{n}$ big enough (that is, $n_{\alpha}$ big enough for any $\alpha\in\Omega$), $\dd$ is a positive root and therefore, by Theorem \ref{Kac},
\[
 A_{Q_{\underline{n}},\dd}(0)=\dim\mathfrak{g}_{Q_{\underline{n}}}[\dd]=\dim\mathfrak{n}_{Q_{\underline{n}}}^+[\dd].
\]
where $\mathfrak{g}_{Q_{\underline{n}}}$ is the Kac-Moody algebra asociated to $Q_{\underline{n}}$ and $\mathfrak{g}_{Q_{\underline{n}}}=\mathfrak{n}_{Q_{\underline{n}}}^-\oplus\mathfrak{h}_{Q_{\underline{n}}}\oplus\mathfrak{n}_{Q_{\underline{n}}}^+$ its triangular decomposition. By \cite[Theorem 9.11]{MR1104219}, the $\N^I$-graded Lie algebra $\mathfrak{n}_{Q_{\underline{n}}}^+$ has (Chevalley) generators $E_i$, $i\in I$ verifying Serre's relations: if $i\neq j$, and $a_{ij}=-\sum_{\alpha : i\rightarrow j}n_{\alpha}-\sum_{\alpha:j\rightarrow i}n_{\alpha}$
\[
 \ad(E_i)^{1-a_{ij}}(E_j)=0
\]
This relation is homogeneous of degree $e_j+(1-a_{ij})e_i$. Therefore, when $\underline{n}\rightarrow \underline{m}$, $\dim\mathfrak{g}_{Q_{\underline{n}}}[\dd]$ stabilizes and converges to $\dim\mathfrak{g}_{Q_{\underline{m}}}$, where $\mathfrak{g}_{Q_{\underline{m}}}$ is defined the same way as $\mathfrak{g}_{Q_{\underline{n}}}$, except that if $m_{\alpha}=\infty$, $\alpha:i\rightarrow j$, there is no relation between $E_i$ and $E_j$. If $Q$ is complete, that is any two vertices are connected by an arrow, and $\underline{m}=\underline{\infty}$, then $\mathfrak{g}_{Q_{\underline{\infty}}}$ is the free Lie algebra on generators $E_i$, $i\in I$. Its character is given by Witt's dimension formula. For $\dd=\sum_{i\in I}d_ie_i\in\N^I$ and $\lvert\dd\rvert=\sum_{i\in I}d_i$,
\[
 \dim\mathfrak{g}_{Q_{\underline{\infty}}}[\dd]=\frac{1}{\lvert\dd\rvert}\sum_{d\mid d_i, i\in I}\mu(d)\frac{\frac{\lvert\dd\rvert}{d}!}{\prod_{i\in I}\frac{d_i}{d}!}
\]
(see \emph{e.g.} \cite[(2.1.21)]{MR840216}). It is easy to check on the examples that for the generalized Kronecker quivers, see Example \ref{exKronecker}, this formula indeed gives the constant coefficient of $A_{K_{\infty},\dd}$.
\subsection{Rate of convergence}
Theorem \ref{thmtech} also gives the rate of convergence of $A_{Q_{\underline{n}},\dd}$ to $A_{Q_{\underline{m}},\dd}$. To achieve this, let
\[
 \alpha_{\underline{n}}=val(A_{Q_{\underline{n}},\dd}-A_{Q_{\underline{m}},\dd})
\]
be the highest power of $q$ dividing $A_{Q_{\underline{n}},\dd}(q)-A_{Q_{\underline{m}},\dd}(q)$.
\begin{proposition}
 The rate of convergence is linear in $\underline{n}$, that is $\alpha_{\underline{n}}(q)$ is bigger than a linear function of $\underline{n}$ whose limit in any infinite direction is infinite. 
\end{proposition}
In fact, the optimal choice can be determined using the affine linear functions $l_j$ given by Theorem \ref{thmtech}. For example, we expect to have $\alpha_{r}=r-\max(d_1,d_2)+1$ for the generalized Kronecker quiver and $\alpha_g=2g$ for $d=1$, $\alpha_g=2g-1$ for $d\geq 2$ for the multi-loops quivers.

\subsection{Dimension count of cuspidal functions}
In \cite{MR3968896}, Bozec and Schiffmann introduced two families of polynomials, $C_{Q,\dd}(q)$ and $C_{Q,\dd}^{abs}(q)$. When evaluated at $q=\lvert\F_q\rvert$,
\[
 C_{Q,\dd}(q)=\dim\HH_{Q,\F_q}^{cusp}[\dd]
\]
where $\HH_{Q,\F_q}^{cusp}[\dd]$ is the vector space of cuspidal functions of degree $\dd$ in the Hall algebra of $Q$ over $\F_q$ (see \emph{op. cit.} for more details). The polynomials $C_{Q,\dd}^{abs}(q)$ are obtained by inductively inverting Borcherds character formula for Borcherds Lie algebras starting from the formula \cite[(4.3)]{MR3968896}:
\[
 \pi(\ch(U(\tilde{\mathfrak{n}}_Q^{\N})))=\Exp_{q,z}\left(\sum_{\dd>0}A_{Q,\dd}(q)z^{\dd}\right),
\]
where $\tilde{\mathfrak{n}}_{Q}^{\N}$ is the positive part of a conjectural Borcherds Lie algebra whose character is given by the right hand side of the above formula. Moreover, these two families of polynomials are related:
\[
  C_{Q,\dd}^{abs}(q)=C_{Q,\dd}(q) \text{ if $\dd\in\N^I$ , $\langle\dd,\dd\rangle<0$}
\]
\[
\Exp_{z}\left(\sum_{l\geq 1}C_{Q,l\dd}(q)z^l\right)=\Exp_{q,z}\left(\sum_{l\geq 1}C_{Q,l\dd}^{abs}(q)z^{l}\right)\text{ if $\dd\in\N^I$ is indivisible and $\langle\dd,\dd\rangle\leq 0$}.
\]
Combining all of this, Theorem \ref{thmconv} has as a consequence the stabilization of the family of polynomials $C_{Q_{\underline{n}},\dd}(q)$ and $C_{Q_{\underline{n}},\dd}^{abs}(q)$ as we increase the number of arrows, if we assume that $Q$ has no loops. It seems more difficult to conclude if $Q$ has loops although it is reasonable to expect they also verify a stabilization property of the same kind. Computing the valuation of cuspidal polynomials is also a problem for which no answer is known.

\section{Geometric interpretation}
\label{geominterp}
It would be nice to be able to understand geometrically these stabilization properties of Kac polynomials. In this Section, we give some ideas of where to look for such an interpretation.

\subsection{Nakajima quiver varieties} Kac polynomials have an interpretation related to the cohomology of Nakajima quiver varieties. These are symplectic resolutions defined in terms of GIT quotients of moduli spaces of framed representations of a quiver whose link with the representation theory of Kac-Moody algebras has been studied by Nakajima from the 1990s (\cite{MR1302318, MR1604167}). For indivisible dimension vectors, the link between the Kac polynomial and Nakajima quiver varieties is given in \cite{MR2038196} and for any dimension vector in \cite{MR3034296}. Following King (\cite{MR1315461}), for $\lambda\in \Z^I$ and $\dd\in \N^I$ such that $\lambda\cdot\dd:=\sum_{i\in I}\lambda_i\dd_i=0$, a representation $M$ of $Q$ of dimension vector $\dd$ is said to be stable if $\lambda\cdot\dim N<0$ for any subrepresentation $0\neq N\subsetneq M$.
\begin{theorem}[\cite{MR2038196}]
\label{CBtheorem}
 Let $\dd$ be an indivisible dimension vector, $\lambda\in\Z^I$ such that $\lambda\cdot\dd=0$ and $\lambda\cdot\ee\neq 0$ for any $0<\ee<\dd$. Then
 \[
  A_{Q,\dd}(q)=\sum_{i=0}^d\dim H^{2d-2i}(X_s,\C)q^{i}
 \]
where $X_s$ is the moduli space of $\lambda$-stable representations of dimension vector $\dd$ of the preprojective algebra $\Pi_Q$ and $\dim X_s=2d$. (The variety $X_s$ is an example of a Nakajima quiver variety when the framing is trivial).
\end{theorem}
For clarity, we briefly explain the construction of $X_s$. First recall that the preprojective algebra $\Pi_Q$ is the quotient of the path algebra of the doubled quiver $\overline{Q}$ (for each arrow $a$ of $Q$, we add an arrow $a^*$ in the opposite direction) by the two-sided ideal generated by
\[
 \sum_{a\in\Omega}(aa^*-a^*a).
\]
The vector space 
\[
 E_{\overline{Q},\dd}=\bigoplus_{a:i\rightarrow j\in\Omega}\Hom(\C^{\dd_i},\C^{\dd_j})\oplus\bigoplus_{a:i\rightarrow j\in\Omega}\Hom(\C^{\dd_j},\C^{\dd_i})
\]
is naturally acted on by the product of linear groups
\[
 \GL_{\dd}=\prod_{i\in I}\GL_{\dd_i}.
\]
Let $E_{\Pi_Q,\dd}$ be the closed subvariety of $E_{\overline{Q},\dd}$ of elements $(x_a,x_{a^*})_{a\in\Omega}$ such that $\sum_{a\in \Omega}(x_ax_{a^*}-x_{a^*}x_a)=0$. The set-theoretic quotient $E_{\Pi_{Q},\dd}/\GL_{\dd}$ parametrizes the isomorphism classes of representations of $\Pi_Q$ of dimension vector $\dd$. The action of $\GL_{\dd}$ factorizes through $PGL_{\dd}=\GL_{\dd}/\C^*$ and an element $\lambda\in\Z^I$ such that $\lambda\cdot\dd=0$ gives a character of $PGL_{\dd}$:
\[
 \begin{matrix}
  \chi&:&PGL_{\dd}&\rightarrow&\C^*\\
  &&(\overline{g_i}))_{i\in I}&\mapsto&\prod_{i\in I}\det(g_i)^{-\lambda_i}
 \end{matrix}.
\]
This is a linearization of the trivial vector bundle over $E_{\Pi_Q,\dd}$ for the $PGL_{\dd}$-action and it can be checked that the stable points of $E_{\Pi_{Q},\dd}$ are precisely whose $x\in E_{\Pi_{Q},\dd}$ giving a stable representation of $\Pi_Q$. We let $E_{\Pi_Q,\dd}^{\lambda-st}\subset E_{\Pi_Q,\dd}$ be the open subset of $\lambda$-stable elements. The variety $X_s$ is obtained as the GIT quotient $E_{\Pi_Q,\dd}/\!/\!_{\chi}PGL_{\dd}$.

\begin{example}
 For $Q=K_r$, $\dd=(1,1)$, $\lambda=(1,-1)$, a representation $(a_i,b_i)_{1\leq i\leq r}$ of $\Pi_Q$ is stable if and only if $(a_1,\hdots,a_r)\neq 0$. In this case, $X_s$ is the total space of the tautological bundle over the grassmannian of hyperplanes in $\C^r$.  Therefore, the cohomology of $X_s$ is in this case very simple ($H^i(X_s,\C)$ is one dimensional for $i$ even and $0\leq i\leq 2(r-1)$ and zero else). Combining this with Theorem \ref{CBtheorem}, we obtain an explicit formula for the Kac polynomial $A_{K_r,(1,1)}(q)$ from which we obtain a proof of Theorem \ref{thmconv} for the Kronecker quiver and the dimension vector $(1,1)$ (although this case is trivial).

\end{example}

\subsection{Lusztig nilpotent variety}
Kac polynomials appear in the counting of points of Lusztig nilpotent varieties, as shown in \cite{BSV}.
\begin{theorem}
\label{ptsnil}
 For any quiver $Q$,
 \[
  \Exp_{q,z}\left(\frac{1}{1-q^{-1}}\sum_{\dd}A_{\dd}(q^{-1})z^{\dd}\right)=\sum_{\dd}\frac{\lvert\Lambda_{\dd}(\F_q)\rvert}{\lvert G_{\dd}(\F_q)\rvert}q^{\langle\dd,\dd\rangle}z^{\dd}.
 \]

\end{theorem}
The interesting point in this formula is that only non-positive powers of $q$ appear in the left hand side. It implies that $q\mapsto \frac{\lvert\Lambda_{\dd}(\F_q)\rvert}{\lvert G_{\dd}(\F_q)\rvert}q^{\langle\dd,\dd\rangle}$ is a polynomial in $q^{-1}$. The appearance of $A_{Q,\dd}(q^{-1})$ instead of $A_{Q,\dd}(q)$ in Theorem \ref{ptsnil} suggests that Theorem \ref{recipr} concerning the reciprocal of Kac polynomials $q^{\deg A_{Q,\dd}(q)}A_{Q,\dd}(q^{-1})$ should have a geometric interpretation. 

The point count of the zero-level of the moment map also has something to do with Kac polynomials:
\begin{theorem}[{\cite[Formula (1.4)]{BSV}}]
For any quiver $Q$, we have
\[
 \sum_{\dd}\frac{\lvert\mu^{-1}_{\dd}(0)(\F_q)\rvert}{\lvert G_{\dd}(\F_q)\rvert}=\Exp_{q,z}\left(\frac{q}{q-1}\sum_{\dd}A_{\dd}(q)z^{\dd}\right).
\]
\end{theorem}
Because of this formula, we can expect some interpretation of Conjecture $\ref{conjecture}$ and Theorem \ref{thmconv} in terms of some geometric properties of $\mu^{-1}_{\dd}(0)$ when we increase the number of arrows.

\section{Distribution of the coefficients of Kac polynomials}
In this Section, we explore global properties of Kac polynomials of quivers. Motivated by the Hausel--Rodriguez-Villegas paper \cite{MR3364745} and the interpretation of Kac polynomials in terms of Poincaré polynomials of some quiver varieties (at least for indivisible dimension vectors), Theorem \ref{CBtheorem}, we were led to study the distribution of the coefficients of Kac polynomials when we increase the number of arrows. More precisely, to a quiver $Q=(I,\Omega)$ and a dimension vector $\dd\in\N^I$, we plot on a graph $G_{Q,\dd,0}$ the renormalized even coefficients of $A_{Q,\dd}(q)$, that is the points $\left(\frac{2j}{\deg A_{Q,\dd}(q)},\frac{a_{Q,\dd,2j}}{a_{Q,\dd}}\right)$ where $A_{Q,\dd}(q)=q^{N_{\dd}}\sum_{j=0}^{\deg A_{Q,\dd}(q)}a_{Q,\dd,j}q^j$ ($N_{\dd}$ is some nonnegative integer, the valuation of $A_{Q,\dd}$, and $a_{Q,\dd,0}\neq 0$) and $a_{Q,\dd}=\max_{0\leq j\leq \deg A_{Q,\dd}(q)}\limits a_{Q,\dd,j}$ and on $G_{Q,\dd,1}$ the renormalized odd coefficients of the polynomial $A_{Q,\dd}(q)$, which are the points $\left(\frac{2j+1}{\deg A_{Q,\dd}(q)},\frac{a_{Q,\dd,2j+1}}{a_{Q,\dd}}\right)$. We make the following conjecture. By abuse, we also let $G_{Q,\dd, a}$ be the piecewise affine curve obtained by joining the dots of $G_{Q,\dd,a}$ for $a=0,1$. Let $\underline{m}\in\N_{\geq 1}^{I}$. For $u\in \N$, we let $u\underline{m}=(um_{\alpha})_{\alpha\in\Omega}$.

\begin{conj}
\label{conjdistrib}
  The curves $G_{Q_{u\underline{m}},\dd}$ converge to a continuous curve when $u\rightarrow \infty$.
\end{conj}
\begin{remark}
 \begin{enumerate}
  \item The difference with Theorem \ref{thmconv} is that here, we need to fix a direction $\underline{m}$ to have the convergence.
  \item In \cite{MR3364745}, the authors studied the distribution of the Betti numbers of semiprojective hyperkähler varieties by taking a limit involving a different operation on the quiver than ours: we increase the number of arrows without changing the set of vertices while the authors of \emph{op. cit.} consider growing complete graphs with dimension vector one at each vertex, or the tennis-racket quiver with dimension vector $(1,n)$, $n$ being the dimension at the loop-vertex, and let $n$ goes to infinity.
 \end{enumerate}
\end{remark}
To illustrate Conjecture \ref{conjdistrib}, we give the graphs obtained for some specific quivers.

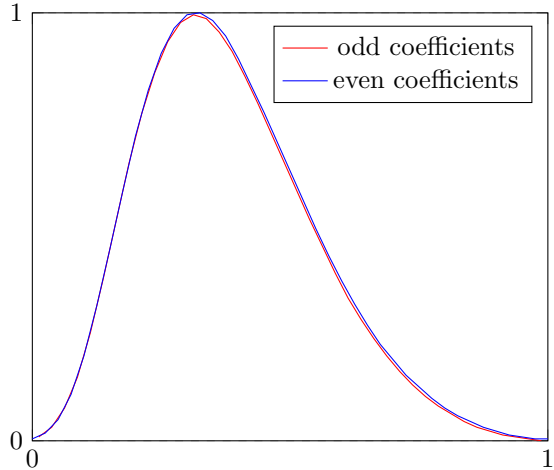
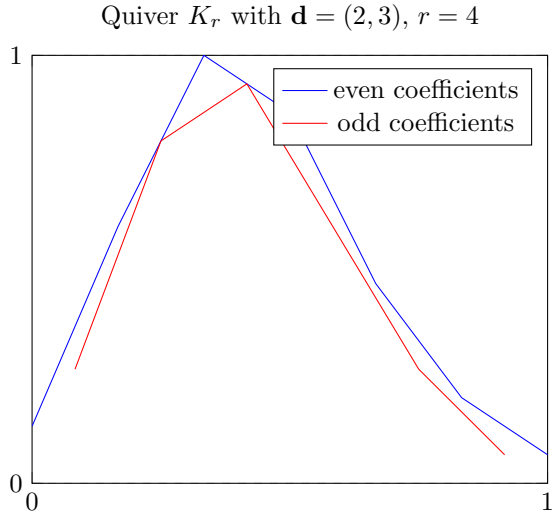
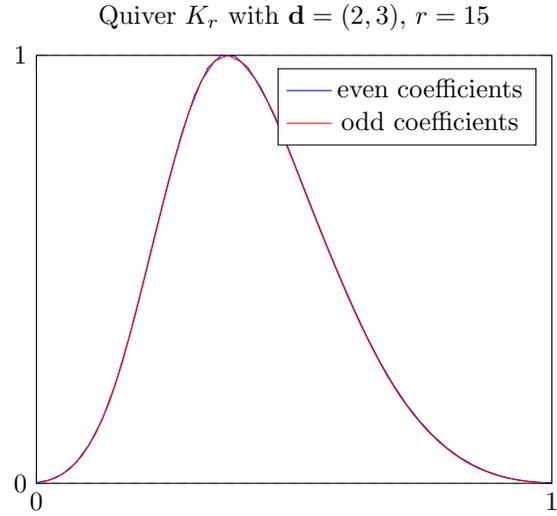
\begin{tikzpicture}
\begin{axis}[
    title={Quiver $S_{g}$ with $d=5$, $g=10$},
    xlabel={},
    ylabel={},
    xmin=0, xmax=1,
    ymin=0, ymax=1,
    xtick={0,1},
    ytick={0,1},
    legend pos=north east,
    ymajorgrids=true,
    grid style=dashed,
]

\addplot[
    color=blue,
    mark=,
    ]
    coordinates {
    (0, 1/1990)(1/90, 1/995)(1/45, 7/1990)(1/30, 17/1990)(2/45, 33/1990)(1/18, 57/1990)(1/15, 91/1990)(7/90, 68/995)(4/45, 97/995)(1/10, 133/995)(1/9, 353/1990)(11/90, 453/1990)(2/15, 563/1990)(13/90, 681/1990)(7/45, 402/995)(1/6, 93/199)(8/45, 1057/1990)(17/90, 591/995)(1/5, 1303/1990)(19/90, 708/995)(2/9, 152/199)(7/30, 323/398)(11/45, 170/199)(23/90, 355/398)(4/15, 1839/1990)(5/18, 946/995)(13/45, 967/995)(3/10, 982/995)(14/45, 991/995)(29/90, 994/995)(1/3, 1983/1990)(31/90, 1969/1990)(16/45, 973/995)(11/30, 383/398)(17/45, 1877/1990)(7/18, 1833/1990)(2/5, 892/995)(37/90, 1731/1990)(19/45, 838/995)(13/30, 809/995)(4/9, 1559/1990)(41/90, 1499/1990)(7/15, 1437/1990)(43/90, 275/398)(22/45, 656/995)(1/2, 624/995)(23/45, 237/398)(47/90, 561/995)(8/15, 1059/1990)(49/90, 997/1990)(5/9, 468/995)(17/30, 438/995)(26/45, 409/995)(53/90, 381/995)(3/5, 707/1990)(11/18, 131/398)(28/45, 303/995)(19/30, 559/1990)(29/45, 103/398)(59/90, 473/1990)(2/3, 433/1990)(61/90, 198/995)(31/45, 361/1990)(7/10, 164/995)(32/45, 297/1990)(13/18, 134/995)(11/15, 241/1990)(67/90, 108/995)(34/45, 193/1990)(23/30, 171/1990)(7/9, 151/1990)(71/90, 133/1990)(4/5, 58/995)(73/90, 101/1990)(37/45, 87/1990)(5/6, 37/995)(38/45, 63/1990)(77/90, 53/1990)(13/15, 22/995)(79/90, 18/995)(8/9, 29/1990)(9/10, 23/1990)(41/45, 9/995)(83/90, 7/995)(14/15, 1/199)(17/18, 7/1990)(43/45, 1/398)(29/30, 3/1990)(44/45, 1/995)(89/90, 1/1990)(1, 0)};
    \addlegendentry{even coefficients}
    
    \addplot[
    color=red,
    mark=,
    ]
    coordinates {
 (1/180, 1/995)
 (1/60, 7/1990)
 (1/36, 17/1990)
 (7/180, 33/1990)
 (1/20, 57/1990)
 (11/180, 91/1990)
 (13/180, 68/995)
 (1/12, 97/995)
 (17/180, 133/995)
 (19/180, 353/1990)
 (7/60, 453/1990)
 (23/180, 563/1990)
 (5/36, 681/1990)
 (3/20, 402/995)
 (29/180, 93/199)
 (31/180, 1057/1990)
 (11/60, 591/995)
 (7/36, 1303/1990)
 (37/180, 708/995)
 (13/60, 152/199)
 (41/180, 323/398)
 (43/180, 170/199)
 (1/4, 355/398)
 (47/180, 1839/1990)
 (49/180, 946/995)
 (17/60, 967/995)
 (53/180, 982/995)
 (11/36, 991/995)
 (19/60, 994/995)
 (59/180, 1983/1990)
 (61/180, 1969/1990)
 (7/20, 973/995)
 (13/36, 383/398)
 (67/180, 1877/1990)
 (23/60, 1833/1990)
 (71/180, 892/995)
 (73/180, 1731/1990)
 (5/12, 838/995)
 (77/180, 809/995)
 (79/180, 1559/1990)
 (9/20, 1499/1990)
 (83/180, 1437/1990)
 (17/36, 275/398)
 (29/60, 656/995)
 (89/180, 624/995)
 (91/180, 237/398)
 (31/60, 561/995)
 (19/36, 1059/1990)
 (97/180, 997/1990)
 (11/20, 468/995)
 (101/180, 438/995)
 (103/180, 409/995)
 (7/12, 381/995)
 (107/180, 707/1990)
 (109/180, 131/398)
 (37/60, 303/995)
 (113/180, 559/1990)
 (23/36, 103/398)
 (13/20, 473/1990)
 (119/180, 433/1990)
 (121/180, 198/995)
 (41/60, 361/1990)
 (25/36, 164/995)
 (127/180, 297/1990)
 (43/60, 134/995)
 (131/180, 241/1990)
 (133/180, 108/995)
 (3/4, 193/1990)
 (137/180, 171/1990)
 (139/180, 151/1990)
 (47/60, 133/1990)
 (143/180, 58/995)
 (29/36, 101/1990)
 (49/60, 87/1990)
 (149/180, 37/995)
 (151/180, 63/1990)
 (17/20, 53/1990)
 (31/36, 22/995)
 (157/180, 18/995)
 (53/60, 29/1990)
 (161/180, 23/1990)
 (163/180, 9/995)
 (11/12, 7/995)
 (167/180, 1/199)
 (169/180, 7/1990)
 (19/20, 1/398)
 (173/180, 3/1990)
 (35/36, 1/995)
 (59/60, 1/1990)
 (179/180, 0)};
    \addlegendentry{odd coefficients}
    
\end{axis}
\end{tikzpicture}
\begin{tikzpicture}
\begin{axis}[
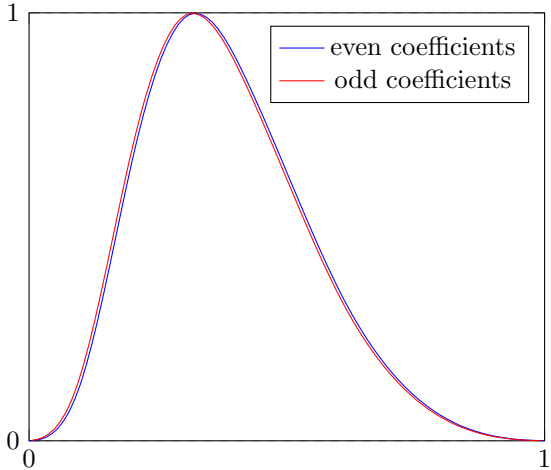
    title={Quiver $S_{g}$ with $d=5$, $g=5$},
    xlabel={},
    ylabel={},
    xmin=0, xmax=1,
    ymin=0, ymax=1,
    xtick={0,1},
    ytick={0,1},
    legend pos=north east,
    ymajorgrids=true,
    grid style=dashed,
]

\addplot[
    color=red,
    mark=,
    ]
    coordinates {
 (1/80, 1/111)
 (3/80, 7/222)
 (1/16, 17/222)
 (7/80, 11/74)
 (9/80, 28/111)
 (11/80, 14/37)
 (13/80, 19/37)
 (3/16, 24/37)
 (17/80, 85/111)
 (19/80, 191/222)
 (21/80, 69/74)
 (23/80, 217/222)
 (5/16, 221/222)
 (27/80, 73/74)
 (29/80, 106/111)
 (31/80, 101/111)
 (33/80, 63/74)
 (7/16, 175/222)
 (37/80, 80/111)
 (39/80, 145/222)
 (41/80, 65/111)
 (43/80, 115/222)
 (9/16, 101/222)
 (47/80, 29/74)
 (49/80, 1/3)
 (51/80, 21/74)
 (53/80, 53/222)
 (11/16, 22/111)
 (57/80, 6/37)
 (59/80, 29/222)
 (61/80, 23/222)
 (63/80, 3/37)
 (13/16, 7/111)
 (67/80, 5/111)
 (69/80, 7/222)
 (71/80, 5/222)
 (73/80, 1/74)
 (15/16, 1/111)
 (77/80, 1/222)
 (79/80, 0)};
    \addlegendentry{odd coefficients}

\addplot[
    color=blue,
    mark=,
    ]
    coordinates {
(0, 1/222)
 (1/40, 2/111)
 (1/20, 11/222)
 (3/40, 4/37)
 (1/10, 22/111)
 (1/8, 35/111)
 (3/20, 33/74)
 (7/40, 43/74)
 (1/5, 79/111)
 (9/40, 91/111)
 (1/4, 67/74)
 (11/40, 107/111)
 (3/10, 221/222)
 (13/40, 1)
 (7/20, 109/111)
 (3/8, 35/37)
 (2/5, 33/37)
 (17/40, 92/111)
 (9/20, 85/111)
 (19/40, 155/222)
 (1/2, 70/111)
 (21/40, 125/222)
 (11/20, 55/111)
 (23/40, 16/37)
 (3/5, 83/222)
 (5/8, 71/222)
 (13/20, 10/37)
 (27/40, 25/111)
 (7/10, 7/37)
 (29/40, 17/111)
 (3/4, 14/111)
 (31/40, 11/111)
 (4/5, 17/222)
 (33/40, 13/222)
 (17/20, 5/111)
 (7/8, 7/222)
 (9/10, 5/222)
 (37/40, 1/74)
 (19/20, 1/111)
 (39/40, 1/222)
 (1, 1/222)
 };
    \addlegendentry{even coefficients}
    
\end{axis}
\end{tikzpicture}

\begin{tikzpicture}
\begin{axis}[
    title={Quiver $K_{r}$ with $\dd=(2,3)$, $r=4$},
    xlabel={},
    ylabel={},
    xmin=0, xmax=1,
    ymin=0, ymax=1,
    xtick={0,1},
    ytick={0,1},
    legend pos=north east,
    ymajorgrids=true,
    grid style=dashed,
]

\addplot[
    color=blue,
    mark=,
    ]
    coordinates {
 (0, 2/15)
 (1/6, 3/5)
 (1/3, 1)
 (1/2, 13/15)
 (2/3, 7/15)
 (5/6, 1/5)
 (1, 1/15)};
    \addlegendentry{even coefficients}

\addplot[
    color=red,
    mark=,
    ]
    coordinates {
(1/12, 4/15)
 (1/4, 4/5)
 (5/12, 14/15)
 (7/12, 3/5)
 (3/4, 4/15)
 (11/12, 1/15)
 };
    \addlegendentry{odd coefficients}
    
\end{axis}
\end{tikzpicture}
\begin{tikzpicture}
\begin{axis}[
    title={Quiver $K_{r}$ with $\dd=(2,3)$, $r=15$},
    xlabel={},
    ylabel={},
    xmin=0, xmax=1,
    ymin=0, ymax=1,
    xtick={0,1},
    ytick={0,1},
    legend pos=north east,
    ymajorgrids=true,
    grid style=dashed,
]

\addplot[
    color=blue,
    mark=,
    ]
    coordinates {
 (0, 2/1323)
 (1/39, 10/1323)
 (2/39, 29/1323)
 (1/13, 64/1323)
 (4/39, 17/189)
 (5/39, 199/1323)
 (2/13, 103/441)
 (7/39, 64/189)
 (8/39, 202/441)
 (3/13, 772/1323)
 (10/39, 934/1323)
 (11/39, 1081/1323)
 (4/13, 1202/1323)
 (1/3, 1285/1323)
 (14/39, 1)
 (5/13, 1321/1323)
 (16/39, 1285/1323)
 (17/39, 1222/1323)
 (6/13, 1139/1323)
 (19/39, 1042/1323)
 (20/39, 134/189)
 (7/13, 278/441)
 (22/39, 244/441)
 (23/39, 634/1323)
 (8/13, 541/1323)
 (25/39, 454/1323)
 (2/3, 125/441)
 (9/13, 305/1323)
 (28/39, 244/1323)
 (29/39, 64/441)
 (10/13, 148/1323)
 (31/39, 37/441)
 (32/39, 3/49)
 (11/13, 19/441)
 (34/39, 38/1323)
 (35/39, 8/441)
 (12/13, 2/189)
 (37/39, 1/189)
 (38/39, 1/441)
 (1, 1/1323)};
    \addlegendentry{even coefficients}

\addplot[
    color=red,
    mark=,
    ]
    coordinates {
(1/78, 4/1323)
 (1/26, 17/1323)
 (5/78, 43/1323)
 (7/78, 29/441)
 (3/26, 22/189)
 (11/78, 248/1323)
 (1/6, 373/1323)
 (5/26, 523/1323)
 (17/78, 14/27)
 (19/78, 851/1323)
 (7/26, 1007/1323)
 (23/78, 1142/1323)
 (25/78, 415/441)
 (9/26, 1306/1323)
 (29/78, 1)
 (31/78, 1303/1323)
 (11/26, 179/189)
 (35/78, 131/147)
 (37/78, 1088/1323)
 (1/2, 47/63)
 (41/78, 2/3)
 (43/78, 779/1323)
 (15/26, 97/189)
 (47/78, 583/1323)
 (49/78, 493/1323)
 (17/26, 410/1323)
 (53/78, 335/1323)
 (55/78, 10/49)
 (19/26, 214/1323)
 (59/78, 166/1323)
 (61/78, 2/21)
 (21/26, 31/441)
 (5/6, 22/441)
 (67/78, 5/147)
 (23/26, 29/1323)
 (71/78, 17/1323)
 (73/78, 1/147)
 (25/26, 4/1323)
 (77/78, 1/1323)
 };
    \addlegendentry{odd coefficients}
    
\end{axis}
\end{tikzpicture}

\begin{tikzpicture}
\begin{axis}[
    title={Tennis-racket quiver with $\dd=(2,3)$, $n_{\alpha}=10, n_{\beta}=20$},
    xlabel={},
    ylabel={},
    xmin=0, xmax=1,
    ymin=0, ymax=1,
    xtick={0,1},
    ytick={0,1},
    legend pos=north east,
    ymajorgrids=true,
    grid style=dashed,
]

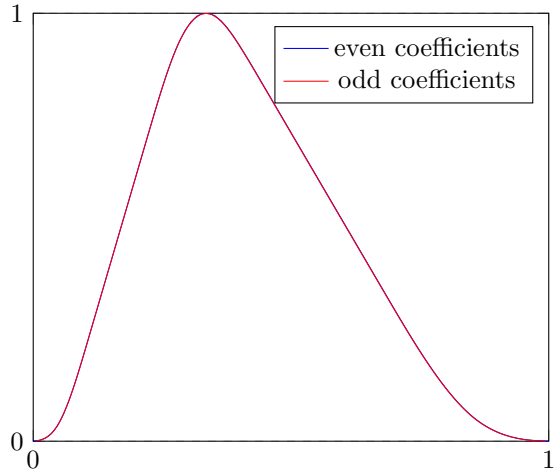
\addplot[
    color=blue,
    mark=,
    ]
    coordinates {
 (0, 1/5985)
 (1/85, 11/5985)
 (2/85, 2/285)
 (3/85, 106/5985)
 (4/85, 43/1197)
 (1/17, 379/5985)
 (6/85, 592/5985)
 (7/85, 842/5985)
 (8/85, 1117/5985)
 (9/85, 281/1197)
 (2/17, 113/399)
 (11/85, 397/1197)
 (12/85, 65/171)
 (13/85, 3/7)
 (14/85, 571/1197)
 (3/17, 629/1197)
 (16/85, 229/399)
 (1/5, 745/1197)
 (18/85, 803/1197)
 (19/85, 41/57)
 (4/17, 1531/1995)
 (21/85, 232/285)
 (22/85, 5134/5985)
 (23/85, 5371/5985)
 (24/85, 1115/1197)
 (5/17, 1913/1995)
 (26/85, 1954/1995)
 (27/85, 5944/5985)
 (28/85, 1)
 (29/85, 1)
 (6/17, 313/315)
 (31/85, 653/665)
 (32/85, 1927/1995)
 (33/85, 1133/1197)
 (2/5, 5536/5985)
 (7/17, 120/133)
 (36/85, 1753/1995)
 (37/85, 341/399)
 (38/85, 142/171)
 (39/85, 965/1197)
 (8/17, 104/133)
 (41/85, 907/1197)
 (42/85, 878/1197)
 (43/85, 283/399)
 (44/85, 820/1197)
 (9/17, 113/171)
 (46/85, 254/399)
 (47/85, 733/1197)
 (48/85, 704/1197)
 (49/85, 75/133)
 (10/17, 34/63)
 (3/5, 617/1197)
 (52/85, 28/57)
 (53/85, 559/1197)
 (54/85, 530/1197)
 (11/17, 167/399)
 (56/85, 472/1197)
 (57/85, 443/1197)
 (58/85, 46/133)
 (59/85, 55/171)
 (12/17, 1781/5985)
 (61/85, 1639/5985)
 (62/85, 1499/5985)
 (63/85, 454/1995)
 (64/85, 1229/5985)
 (13/17, 220/1197)
 (66/85, 976/5985)
 (67/85, 286/1995)
 (4/5, 746/5985)
 (69/85, 641/5985)
 (14/17, 544/5985)
 (71/85, 13/171)
 (72/85, 25/399)
 (73/85, 61/1197)
 (74/85, 244/5985)
 (15/17, 64/1995)
 (76/85, 148/5985)
 (77/85, 37/1995)
 (78/85, 9/665)
 (79/85, 1/105)
 (16/17, 2/315)
 (81/85, 8/1995)
 (82/85, 2/855)
 (83/85, 1/855)
 (84/85, 1/1995)
 (1, 1/5985)};
    \addlegendentry{even coefficients}

\addplot[
    color=red,
    mark=,
    ]
    coordinates {
(1/170, 4/5985)
 (3/170, 23/5985)
 (1/34, 23/1995)
 (7/170, 22/855)
 (9/170, 58/1197)
 (11/170, 32/399)
 (13/170, 713/5985)
 (3/34, 977/5985)
 (1/10, 4/19)
 (19/170, 310/1197)
 (21/170, 368/1197)
 (23/170, 142/399)
 (5/34, 484/1197)
 (27/170, 542/1197)
 (29/170, 200/399)
 (31/170, 94/171)
 (33/170, 716/1197)
 (7/34, 86/133)
 (37/170, 832/1197)
 (39/170, 1483/1995)
 (41/170, 4733/5985)
 (43/170, 556/665)
 (9/34, 5254/5985)
 (47/170, 365/399)
 (49/170, 5659/5985)
 (3/10, 1934/1995)
 (53/170, 656/665)
 (11/34, 1193/1197)
 (57/170, 1)
 (59/170, 1193/1197)
 (61/170, 394/399)
 (63/170, 1942/1995)
 (13/34, 43/45)
 (67/170, 373/399)
 (69/170, 5462/5985)
 (71/170, 5323/5985)
 (73/170, 148/171)
 (15/34, 53/63)
 (77/170, 326/399)
 (79/170, 949/1197)
 (81/170, 920/1197)
 (83/170, 99/133)
 (1/2, 862/1197)
 (87/170, 119/171)
 (89/170, 268/399)
 (91/170, 775/1197)
 (93/170, 746/1197)
 (19/34, 239/399)
 (97/170, 688/1197)
 (99/170, 659/1197)
 (101/170, 10/19)
 (103/170, 601/1197)
 (21/34, 572/1197)
 (107/170, 181/399)
 (109/170, 514/1197)
 (111/170, 485/1197)
 (113/170, 8/21)
 (23/34, 61/171)
 (117/170, 398/1197)
 (7/10, 1846/5985)
 (121/170, 1703/5985)
 (123/170, 1562/5985)
 (25/34, 1424/5985)
 (127/170, 1289/5985)
 (129/170, 386/1995)
 (131/170, 344/1995)
 (133/170, 911/5985)
 (27/34, 796/5985)
 (137/170, 688/5985)
 (139/170, 587/5985)
 (141/170, 26/315)
 (143/170, 82/1197)
 (29/34, 67/1197)
 (147/170, 6/133)
 (149/170, 214/5985)
 (151/170, 166/5985)
 (9/10, 2/95)
 (31/34, 31/1995)
 (157/170, 22/1995)
 (159/170, 1/133)
 (161/170, 29/5985)
 (163/170, 17/5985)
 (33/34, 1/665)
 (167/170, 4/5985)
 (169/170, 1/5985)
 };
    \addlegendentry{odd coefficients}
    
\end{axis}
\end{tikzpicture}
\begin{tikzpicture}
\begin{axis}[
    title={Tennis-racket quiver with $\dd=(2,3)$, $n_{\alpha}=n_{\beta}=10$},
    xlabel={},
    ylabel={},
    xmin=0, xmax=1,
    ymin=0, ymax=1,
    xtick={0,1},
    ytick={0,1},
    legend pos=north east,
    ymajorgrids=true,
    grid style=dashed,
]

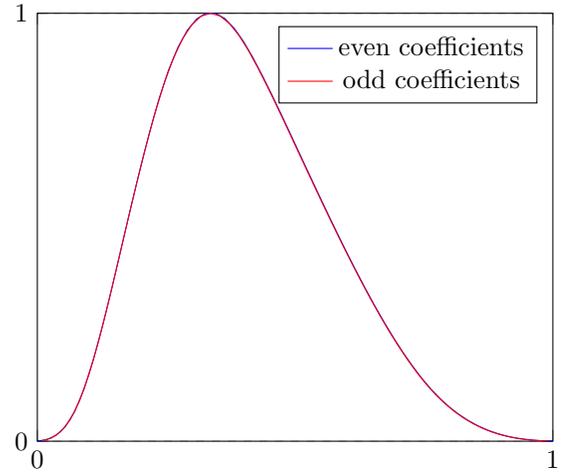
\addplot[
    color=blue,
    mark=,
    ]
    coordinates {
 (0, 1/3085)
 (1/55, 11/3085)
 (2/55, 42/3085)
 (3/55, 106/3085)
 (4/55, 43/617)
 (1/11, 379/3085)
 (6/55, 592/3085)
 (7/55, 842/3085)
 (8/55, 1117/3085)
 (9/55, 281/617)
 (2/11, 1693/3085)
 (1/5, 1972/3085)
 (12/55, 2234/3085)
 (13/55, 2471/3085)
 (14/55, 535/617)
 (3/11, 2839/3085)
 (16/55, 2962/3085)
 (17/55, 3044/3085)
 (18/55, 1)
 (19/55, 1)
 (4/11, 3047/3085)
 (21/55, 2977/3085)
 (2/5, 2881/3085)
 (23/55, 553/617)
 (24/55, 2636/3085)
 (5/11, 500/617)
 (26/55, 2359/3085)
 (27/55, 443/617)
 (28/55, 414/617)
 (29/55, 385/617)
 (6/11, 1781/3085)
 (31/55, 1639/3085)
 (32/55, 1499/3085)
 (3/5, 1362/3085)
 (34/55, 1229/3085)
 (7/11, 220/617)
 (36/55, 976/3085)
 (37/55, 858/3085)
 (38/55, 746/3085)
 (39/55, 641/3085)
 (8/11, 544/3085)
 (41/55, 91/617)
 (42/55, 75/617)
 (43/55, 61/617)
 (4/5, 244/3085)
 (9/11, 192/3085)
 (46/55, 148/3085)
 (47/55, 111/3085)
 (48/55, 81/3085)
 (49/55, 57/3085)
 (10/11, 38/3085)
 (51/55, 24/3085)
 (52/55, 14/3085)
 (53/55, 7/3085)
 (54/55, 3/3085)
 (1, 1/3085)};
    \addlegendentry{even coefficients}

\addplot[
    color=red,
    mark=,
    ]
    coordinates {
(1/110, 4/3085)
 (3/110, 23/3085)
 (1/22, 69/3085)
 (7/110, 154/3085)
 (9/110, 58/617)
 (1/10, 96/617)
 (13/110, 713/3085)
 (3/22, 977/3085)
 (17/110, 252/617)
 (19/110, 1549/3085)
 (21/110, 1833/3085)
 (23/110, 2104/3085)
 (5/22, 2354/3085)
 (27/110, 515/617)
 (29/110, 2759/3085)
 (31/110, 2902/3085)
 (3/10, 3004/3085)
 (7/22, 613/617)
 (37/110, 1)
 (39/110, 613/617)
 (41/110, 602/617)
 (43/110, 2926/3085)
 (9/22, 2819/3085)
 (47/110, 539/617)
 (49/110, 2562/3085)
 (51/110, 2423/3085)
 (53/110, 456/617)
 (1/2, 427/617)
 (57/110, 398/617)
 (59/110, 1846/3085)
 (61/110, 1703/3085)
 (63/110, 1562/3085)
 (13/22, 1424/3085)
 (67/110, 1289/3085)
 (69/110, 1158/3085)
 (71/110, 1032/3085)
 (73/110, 911/3085)
 (15/22, 796/3085)
 (7/10, 688/3085)
 (79/110, 587/3085)
 (81/110, 494/3085)
 (83/110, 82/617)
 (17/22, 67/617)
 (87/110, 54/617)
 (89/110, 214/3085)
 (91/110, 166/3085)
 (93/110, 126/3085)
 (19/22, 93/3085)
 (97/110, 66/3085)
 (9/10, 9/617)
 (101/110, 29/3085)
 (103/110, 17/3085)
 (21/22, 9/3085)
 (107/110, 4/3085)
 (109/110, 1/3085)
 };
    \addlegendentry{odd coefficients}
    
\end{axis}
\end{tikzpicture}
\begin{center}
\begin{tikzpicture}[scale=1]
\begin{axis}[
    height=14cm,
    width=14cm,
    title={Tennis-racket quiver with $\dd=(3,2)$, even coefficients},
    xlabel={},
    ylabel={},
    xmin=0, xmax=1.3,
    ymin=0, ymax=1.3,
    xtick={0,1},
    ytick={0,1},
    legend pos=north east,
    ymajorgrids=true,
    grid style=dashed,
]

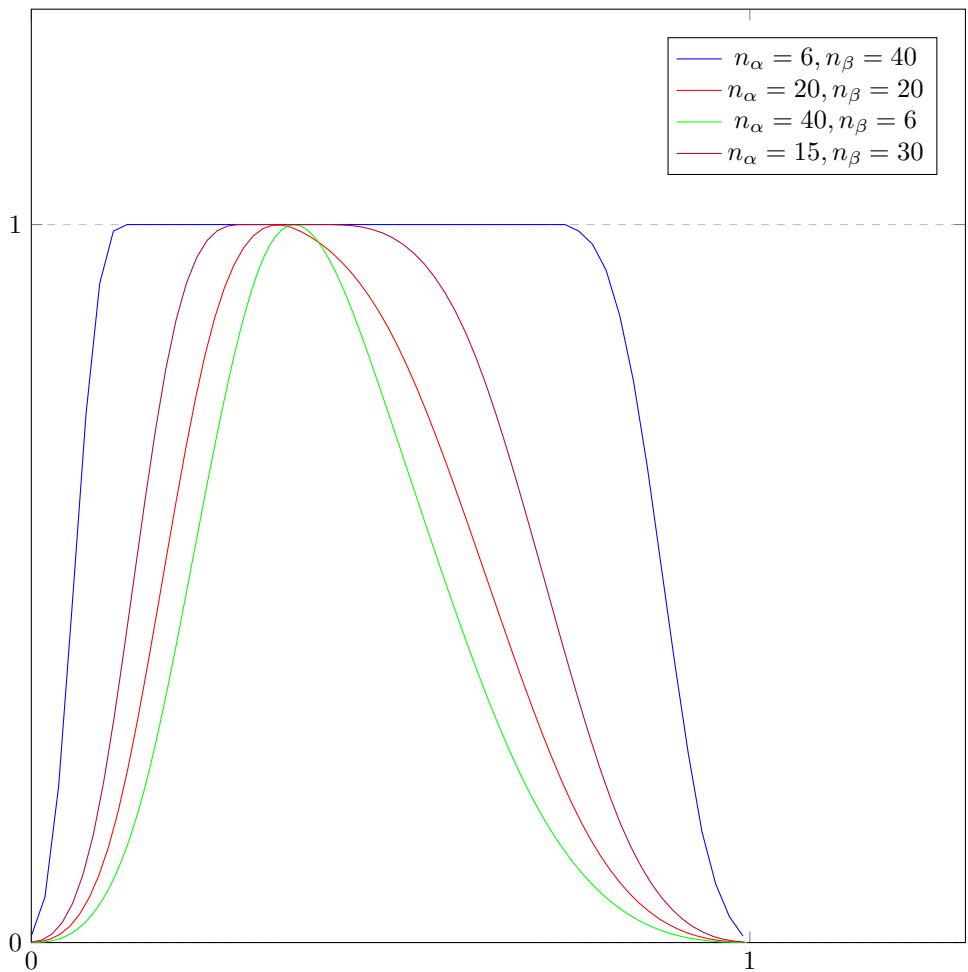
\addplot[
    color=blue,
    mark=,
    ]
    coordinates {
 (0, 1/110)
 (2/105, 7/110)
 (4/105, 12/55)
 (2/35, 26/55)
 (8/105, 81/110)
 (2/21, 101/110)
 (4/35, 109/110)
 (2/15, 1)
 (16/105, 1)
 (6/35, 1)
 (4/21, 1)
 (22/105, 1)
 (8/35, 1)
 (26/105, 1)
 (4/15, 1)
 (2/7, 1)
 (32/105, 1)
 (34/105, 1)
 (12/35, 1)
 (38/105, 1)
 (8/21, 1)
 (2/5, 1)
 (44/105, 1)
 (46/105, 1)
 (16/35, 1)
 (10/21, 1)
 (52/105, 1)
 (18/35, 1)
 (8/15, 1)
 (58/105, 1)
 (4/7, 1)
 (62/105, 1)
 (64/105, 1)
 (22/35, 1)
 (68/105, 1)
 (2/3, 1)
 (24/35, 1)
 (74/105, 1)
 (76/105, 1)
 (26/35, 1)
 (16/21, 109/110)
 (82/105, 107/110)
 (4/5, 103/110)
 (86/105, 48/55)
 (88/105, 43/55)
 (6/7, 73/110)
 (92/105, 29/55)
 (94/105, 43/110)
 (32/35, 29/110)
 (14/15, 17/110)
 (20/21, 9/110)
 (34/35, 2/55)
 (104/105, 1/110)};
    \addlegendentry{$n_{\alpha}=6, n_{\beta}=40$}

\addplot[
    color=red,
    mark=,
    ]
    coordinates {
(0, 1/4878)
 (2/149, 7/4878)
 (4/149, 4/813)
 (6/149, 19/1626)
 (8/149, 37/1626)
 (10/149, 32/813)
 (12/149, 305/4878)
 (14/149, 455/4878)
 (16/149, 36/271)
 (18/149, 889/4878)
 (20/149, 589/2439)
 (22/149, 251/813)
 (24/149, 931/2439)
 (26/149, 745/1626)
 (28/149, 2615/4878)
 (30/149, 997/1626)
 (32/149, 1676/2439)
 (34/149, 1844/2439)
 (36/149, 1994/2439)
 (38/149, 4241/4878)
 (40/149, 2222/2439)
 (42/149, 767/813)
 (44/149, 1573/1626)
 (46/149, 800/813)
 (48/149, 2425/2439)
 (50/149, 4873/4878)
 (52/149, 2437/2439)
 (54/149, 2429/2439)
 (56/149, 4829/4878)
 (58/149, 2396/2439)
 (60/149, 2374/2439)
 (62/149, 2348/2439)
 (64/149, 515/542)
 (66/149, 4565/4878)
 (68/149, 1495/1626)
 (70/149, 2197/2439)
 (72/149, 2146/2439)
 (74/149, 2089/2439)
 (76/149, 4051/4878)
 (78/149, 652/813)
 (80/149, 209/271)
 (82/149, 1801/2439)
 (84/149, 1717/2439)
 (86/149, 3259/4878)
 (88/149, 171/271)
 (90/149, 2893/4878)
 (92/149, 2705/4878)
 (94/149, 2515/4878)
 (96/149, 775/1626)
 (98/149, 356/813)
 (100/149, 1949/4878)
 (102/149, 883/2439)
 (104/149, 794/2439)
 (106/149, 236/813)
 (108/149, 626/2439)
 (110/149, 1097/4878)
 (112/149, 476/2439)
 (114/149, 91/542)
 (116/149, 233/1626)
 (118/149, 197/1626)
 (120/149, 55/542)
 (122/149, 205/2439)
 (124/149, 335/4878)
 (126/149, 15/271)
 (128/149, 107/2439)
 (130/149, 83/2439)
 (132/149, 7/271)
 (134/149, 31/1626)
 (136/149, 11/813)
 (138/149, 5/542)
 (140/149, 29/4878)
 (142/149, 17/4878)
 (144/149, 1/542)
 (146/149, 2/2439)
 (148/149, 1/4878)
 };
    \addlegendentry{$n_{\alpha}=20, n_{\beta}=20$}
    
\addplot[
    color=green,
    mark=,
    ]
    coordinates {
    (0, 1/31405)
 (2/241, 7/31405)
 (4/241, 24/31405)
 (6/241, 57/31405)
 (8/241, 111/31405)
 (10/241, 192/31405)
 (12/241, 304/31405)
 (14/241, 452/31405)
 (16/241, 641/31405)
 (18/241, 175/6281)
 (20/241, 1159/31405)
 (22/241, 1498/31405)
 (24/241, 1896/31405)
 (26/241, 2358/31405)
 (28/241, 2889/31405)
 (30/241, 3493/31405)
 (32/241, 835/6281)
 (34/241, 988/6281)
 (36/241, 5792/31405)
 (38/241, 6736/31405)
 (40/241, 7772/31405)
 (42/241, 8889/31405)
 (44/241, 916/2855)
 (46/241, 11322/31405)
 (48/241, 2523/6281)
 (50/241, 13944/31405)
 (52/241, 15298/31405)
 (54/241, 303/571)
 (56/241, 18034/31405)
 (58/241, 19394/31405)
 (60/241, 20733/31405)
 (62/241, 4408/6281)
 (64/241, 23304/31405)
 (66/241, 24513/31405)
 (68/241, 25656/31405)
 (70/241, 26722/31405)
 (72/241, 27699/31405)
 (74/241, 28576/31405)
 (76/241, 29342/31405)
 (78/241, 5997/6281)
 (80/241, 30502/31405)
 (82/241, 30898/31405)
 (84/241, 31177/31405)
 (86/241, 31344/31405)
 (88/241, 31404/31405)
 (90/241, 31362/31405)
 (92/241, 6245/6281)
 (94/241, 6200/6281)
 (96/241, 30693/31405)
 (98/241, 30311/31405)
 (100/241, 29861/31405)
 (102/241, 29349/31405)
 (104/241, 28782/31405)
 (106/241, 28167/31405)
 (108/241, 5502/6281)
 (110/241, 2438/2855)
 (112/241, 26098/31405)
 (114/241, 25356/31405)
 (116/241, 24599/31405)
 (118/241, 23834/31405)
 (120/241, 23063/31405)
 (122/241, 22288/31405)
 (124/241, 4302/6281)
 (126/241, 4146/6281)
 (128/241, 3990/6281)
 (130/241, 19171/31405)
 (132/241, 18394/31405)
 (134/241, 17621/31405)
 (136/241, 16853/31405)
 (138/241, 16091/31405)
 (140/241, 15337/31405)
 (142/241, 14592/31405)
 (144/241, 13857/31405)
 (146/241, 1194/2855)
 (148/241, 12424/31405)
 (150/241, 11728/31405)
 (152/241, 11048/31405)
 (154/241, 2077/6281)
 (156/241, 1948/6281)
 (158/241, 1823/6281)
 (160/241, 8511/31405)
 (162/241, 7929/31405)
 (164/241, 7371/31405)
 (166/241, 6838/31405)
 (168/241, 6331/31405)
 (170/241, 1170/6281)
 (172/241, 5394/31405)
 (174/241, 4962/31405)
 (176/241, 414/2855)
 (178/241, 379/2855)
 (180/241, 346/2855)
 (182/241, 63/571)
 (184/241, 629/6281)
 (186/241, 569/6281)
 (188/241, 513/6281)
 (190/241, 2304/31405)
 (192/241, 2061/31405)
 (194/241, 1836/31405)
 (196/241, 148/2855)
 (198/241, 1436/31405)
 (200/241, 252/6281)
 (202/241, 1099/31405)
 (204/241, 952/31405)
 (206/241, 819/31405)
 (208/241, 699/31405)
 (210/241, 591/31405)
 (212/241, 9/571)
 (214/241, 82/6281)
 (216/241, 67/6281)
 (218/241, 54/6281)
 (220/241, 214/31405)
 (222/241, 166/31405)
 (224/241, 126/31405)
 (226/241, 93/31405)
 (228/241, 6/2855)
 (230/241, 9/6281)
 (232/241, 29/31405)
 (234/241, 17/31405)
 (236/241, 9/31405)
 (238/241, 4/31405)
 (240/241, 1/31405)
};
\addlegendentry{$n_{\alpha}=40, n_{\beta}=6$}

\addplot[
    color=purple,
    mark=,
    ]
    coordinates {
(0, 1/2030)
 (2/139, 1/290)
 (4/139, 12/1015)
 (6/139, 57/2030)
 (8/139, 111/2030)
 (10/139, 96/1015)
 (12/139, 61/406)
 (14/139, 227/1015)
 (16/139, 317/1015)
 (18/139, 417/1015)
 (20/139, 149/290)
 (22/139, 1251/2030)
 (24/139, 1447/2030)
 (26/139, 162/203)
 (28/139, 176/203)
 (30/139, 932/1015)
 (32/139, 1937/2030)
 (34/139, 397/406)
 (36/139, 2013/2030)
 (38/139, 1013/1015)
 (40/139, 1)
 (42/139, 1)
 (44/139, 1)
 (46/139, 1)
 (48/139, 1)
 (50/139, 1)
 (52/139, 1)
 (54/139, 1)
 (56/139, 1)
 (58/139, 1)
 (60/139, 2029/2030)
 (62/139, 2027/2030)
 (64/139, 289/290)
 (66/139, 144/145)
 (68/139, 1003/1015)
 (70/139, 996/1015)
 (72/139, 1973/2030)
 (74/139, 1949/2030)
 (76/139, 1919/2030)
 (78/139, 941/1015)
 (80/139, 919/1015)
 (82/139, 893/1015)
 (84/139, 345/406)
 (86/139, 331/406)
 (88/139, 788/1015)
 (90/139, 1489/2030)
 (92/139, 698/1015)
 (94/139, 649/1015)
 (96/139, 598/1015)
 (98/139, 78/145)
 (100/139, 141/290)
 (102/139, 63/145)
 (104/139, 779/2030)
 (106/139, 97/290)
 (108/139, 583/2030)
 (110/139, 17/70)
 (112/139, 41/203)
 (114/139, 67/406)
 (116/139, 27/203)
 (118/139, 107/1015)
 (120/139, 83/1015)
 (122/139, 9/145)
 (124/139, 93/2030)
 (126/139, 33/1015)
 (128/139, 9/406)
 (130/139, 1/70)
 (132/139, 17/2030)
 (134/139, 9/2030)
 (136/139, 2/1015)
 (138/139, 1/2030)
 };
 \addlegendentry{$n_{\alpha}=15, n_{\beta}=30$}

\end{axis}
\end{tikzpicture}
\end{center}
The last conjecture follows from the observation of the various graphs above.
\begin{conj}
The sequences of even and odd Betti numbers are unimodal (that is, increasing and then decreasing)
\end{conj}
\begin{remark}
This is a frequent phenomenon having a strong relationship with the Hard Lefschetz Theorem (at least when we consider the Betti numbers of some varieties), see \cite{MR3364745}.
\end{remark}

\appendix
\section{Code for the computations of Kac polynomials}
In this section, we give the entire code used to obtained the results of Example \ref{exKronecker}, Section \ref{multloop} and Section \ref{arrowloop} using SageMath. For the sake of practicality, the code is available on the author's webpage \url{https://www.imo.universite-paris-saclay.fr/~hennecart/}.

\subsection{Explanation for the code}
We use the formulas in \cite{MR1752774}.

If $r\geq 0$, we let $\phi_r(q)=\prod_{j=1}^r(1-q^j)$ and for a partition $\lambda$, $b_{\lambda}(q)=\prod_{i\geq 1}\phi_{n_i}(q)$ if $\lambda=(1^{n_1}2^{n_2}\hdots)$. Note that if $\lambda'$ is the partition conjugate to $\lambda$, one has $n_i=\lambda'_{i}-\lambda'_{i+1}$.

If $\lambda,\mu$ are two partitions, we defined in Theorem \ref{formulaHua} the bilinear product $\langle\lambda,\mu\rangle$. We use another expression for it. If $\lambda'$ (resp. $\mu'$) is the partition conjugate to $\lambda$ (resp. $\mu$), we have
\[
 \langle\lambda,\mu\rangle=\sum_{i\geq 1}\lambda_i'\mu_i'.
\]

We define
\[
 P((x_i)_{i\in I},q)=\sum_{\pi=(\pi^i)\in\mathscr{P}^I}\frac{\prod_{i\rightarrow j\in \Omega}q^{\langle\pi^i,\pi^j\rangle}}{\prod_{i\in I}q^{\langle\pi^i,\pi^i\rangle}b_{\pi^i}(q^{-1})}x^{\lvert\pi\rvert}.
\]

The rational functions of $q$, $H(\dd,q)$, $\dd\in\N^I\setminus\{0\}$ are defined by the equality:
\[
 \log\left(P((x_i)_{i\in I},q)\right)=\sum_{\dd\in\N^I\setminus\{0\}}\frac{H(\dd,q)x^{\dd}}{\overline{\dd}}
\]
where $\log$ is defined by $\log(1-X)=-\sum_{i\geq 1}\frac{X^i}{i}$ and for $\dd\in\N^I$, we let $\overline{\dd}=\gcd(d_i,i\in I)$. 
\begin{theorem}[{\cite[Theorem 4.6]{MR1752774}}]
 For any $\dd\in\N^I\setminus\{0\}$, we have
 \[
  A_{\dd}(q)=\frac{q-1}{\overline{\dd}}\sum_{d\mid \alpha}\mu(d)H\left(\frac{\dd}{d},q^d\right).
 \]
\end{theorem}

\subsection{The Kronecker quivers}
We give here the code for Example \ref{exKronecker}.

\begin{verbatim}
R.<q>=PolynomialRing(QQ)
K=FractionField(R)
#We work with the Kronecker quiver with r arrows from the vertex 1 to the vertex 2.
S.<x1,x2>=PowerSeriesRing(K)
 
def produitbilin(l1,l2): #returns the bilinear products of the partitions l1,l2
    r=0
    l1c=Partition(l1).conjugate()
    l2c=Partition(l2).conjugate()
    l=min(len(l2c),len(l1c))
    for i in range(l):
        r=r+l1c[i]*l2c[i]
    return(r)
    
def phi(r): #returns \phi_r(q^{-1})
    p=1
    for i in range(1,r+1):
        p=p*(1-q^(-i))
    return(p)
    
def b(l): #l is a partition
    lc=Partition(l).conjugate()+[0]
    r=1
    for i in range(len(lc)-1):
        r=r*phi(lc[i]-lc[i+1])
    return(r)
    
def polynomeP(d1,d2,r): #returns P(x_1,x_2,q) for the quiver K_r
    P=S(0)
    for i in range(d1+1):
        for j in range(d2+1):
            for p1 in Partitions(i).list():
                for p2 in Partitions(j).list():
                    P=P+((q^(r*produitbilin(p1,p2)))/
                         (q^(produitbilin(p1,p1
                         +produitbilin(p2,p2))*b(p1)*b(p2)))*(x1^i)*(x2^j)
    return(P)
    
def serieH(d1,d2,N,r): #returns the generating series of the polynomials 
                       #H(\dd,q) given by log(P(x_1,x_2,q))
    s=0
    p=polynomeP(d1,d2,r)
    for i in range(1,N+1):
        s=s-(1-p)^i/i
    return(s)
    
def polyH(N,r,d1,d2): #returns the polynomial H((d_1,d_2),q)
    ser=serieH(d1,d2,N,r)
    d=gcd(d1,d2)
    pol=d*(1/(ZZ(d1).factorial()))*(1/(ZZ(d2).factorial()))*(
        (S((S(ser).derivative(x1,ZZ(d1))).derivative(x2,ZZ(d2)))).constant_coefficient())
    return(pol)
    
def KacPolA(d1,d2,r): #returns the Kac polynomial of K_r for the dimension vector (d_1,d_2).
    N=d1+d2
    d=gcd(d1,d2)
    A=0
    div=d.divisors()
    for i in div:
        A=A+moebius(i)*(polyH(N,r,ZZ(d1/i),ZZ(d2/i))(q^i))
    A=A*((q-1)/d)
    return(A)
    
for i in range(1,5):print(R(KacPolA(2,2,i))) #an example
\end{verbatim}

\subsection{The $g$-loop quivers}
We give here the code for Section \ref{multloop}. Some of it is identical to that in the previous Section, but we reproduce it for the convenience of the reader.

\begin{verbatim}
R.<q>=PolynomialRing(QQ)
K=FractionField(R)
S.<x1,x2>=PowerSeriesRing(K)
#We work with the quiver S_g

def polynomePS(d1,g):
    P=S(0)
    for i in range(d1+1):
        for p1 in Partitions(i).list():
            P=P+((q^(g*produitbilin(p1,p1)))/(q^(produitbilin(p1,p1))*b(p1)))*(x1^i)
    return(P)
    
def serieHS(d1,N,g):
    s=0
    p=polynomePS(d1,g)
    for i in range(1,N+1):
        s=s-(1-p)^i/i
    return(s)
    
def polyHS(N,g,d1):
    ser=serieHS(d1,N,g)
    d=d1
    pol=d*(1/(ZZ(d1).factorial()))*((S((S(ser).derivative(x
    ,ZZ(d1))))).constant_coefficient())
    return(pol)
    
def KacPolAS(d1,g):
    N=d1
    d=d1
    A=0
    div=d.divisors()
    for i in div:
        A=A+moebius(i)*(polyHS(N,g,ZZ(d1/i))(q^i))
    A=A*((q-1)/d)
    return(A)
    
for r in range(1,6): print(KacPolAS(2,g)) 
\end{verbatim}

\subsection{The Tennis Racket quiver}
We give the code for Section \ref{arrowloop}.
\begin{verbatim}
R.<q>=PolynomialRing(QQ)
K=FractionField(R)
S.<x1,x2>=PowerSeriesRing(K)
 
def produitbilin(l1,l2):
   r=0
   l1c=Partition(l1).conjugate()
   l2c=Partition(l2).conjugate()
   l=min(len(l2c),len(l1c))
   for i in range(l):
       r=r+l1c[i]*l2c[i]
   return(r)
def phi(r):
    p=1
    for i in range(1,r+1):
        p=p*(1-q^(-i))
    return(p)
    
def b(l):
    lc=Partition(l).conjugate()+[0]
    r=1
    for i in range(len(lc)-1):
        r=r*phi(lc[i]-lc[i+1])
    return(r)
    
def polynomeP(d1,d2,r,g):
    P=S(0)
    for i in range(d1+1):
        for j in range(d2+1):
            for p1 in Partitions(i).list():
                for p2 in Partitions(j).list():
                    P=P+((q^(r*produitbilin(p1,p2
                    +g*produitbilin(p2,p2)))/(q^(produitbilin(p1,p1
                    +produitbilin(p2,p2))*b(p1)*b(
                    2)))*(x1^i)*(x2^j)
    return(P)
    
def serieH(d1,d2,N,r,g):
    s=0
    p=polynomeP(d1,d2,r,g)
    for i in range(1,N+1):
        s=s-(1-p)^i/i
    return(s)
    
def polyH(N,r,g,d1,d2):
    ser=serieH(d1,d2,N,r,g)
    d=gcd(d1,d2)
    pol=d*(1/(ZZ(d1).factorial()))*(1/(ZZ(d2).factorial()))*(
        (S((S(ser).derivative(x1,ZZ(d1))).derivative(x2,ZZ(d2)))).constant_coefficient())
    return(pol)
    
def KacPolA(d1,d2,r,g):
    N=d1+d2
    d=gcd(d1,d2)
    A=0
    div=d.divisors()
    for i in div:
        A=A+moebius(i)*(polyH(N,r,g,ZZ(d1/i),ZZ(d2/i))(q^i))
    A=A*((q-1)/d)
    return(A)

\end{verbatim}

\section*{Acknowledgements}
The author warmly thanks Olivier Schiffmann for useful comments and corrections concerning this work and the anonymous referees whose comments have allowed considerable improvements and corrections of inacurracies in a previous version.

\end{document}